\documentclass[11pt]{article}
\usepackage{amssymb}
\usepackage[centering,textwidth=150mm,textheight=220mm]{geometry}
\usepackage{mathrsfs}
\usepackage{amsfonts}
\usepackage{amsmath}
\usepackage{cite,xcolor}
\numberwithin{equation}{section}
\allowdisplaybreaks[4]
\newtheorem{Theorem}{Theorem}[section]
\newtheorem{Proposition}[Theorem]{Proposition}

\newtheorem{Lemma}[Theorem]{Lemma}

\def\bbR{\mathbb{R}}
\def\rank{\mathrm{rank}}

\begin{document}


\title{Bilinear Fractional Integral Operators\thanks{%
This work was partially supported by the
National Natural Science Foundation of China (11525104, 11531013, 11761131002, 11801282 and 11926342).}}

\author{Ting Chen\quad and \quad Wenchang Sun\\
\small School of Mathematical Sciences and LPMC,  Nankai University,       Tianjin 300071,  China\\
Emails:\, t.chen@nankai.edu.cn,\quad sunwch@nankai.edu.cn}

\date{}
\maketitle

\begin{abstract}
We study the  bilinear fractional integral considered by
Kenig and Stein, where linear combinations of variables with matrix
coefficients are involved. Under more general settings, we give a complete characterization of the corresponding parameters for which the bilinear fractional integral is bounded from
$L^{p_1}(\bbR^{n_1}) \times L^{p_2}(\bbR^{n_2})$ to $L^q(\bbR^m)$.
\end{abstract}


\section{Introduction and The Main Result}

The multilinear singular integral operators
have been widely studied since
Coifman and Meyer's pioneer work \cite{Coifman1975}.
Christ and Journ\'{e} \cite{Chirst1987},
Grafakos and Torres \cite{Grafakos2002} developed the
theory of multilinear Calder\'on-Zygmund operators.
Lacey and Thiele \cite{Lacey1997,Lacey1999}
proved the boundedness of the bilinear Hilbert transform.

In this paper, we focus on the bilinear fractional integral studied
by Kenig and Stein \cite{KenigStein1999},
Grafakos and Kalton \cite{Grafakos1992,GrafakosKalton2001},
Grafakos and Lynch \cite{GrafakosLynch2015}.
Recall that for $0<\lambda<2n$, $f_1\in L^{p_1}(\bbR^n)$ and $f_2\in L^{p_2}(\bbR^n)$, the bilinear
fractional integral of $(f_1,f_2)$ is defined by
\[
   \int_{\bbR^{2n}}
     \frac{f_1(y_1)f_2(y_2) dy_1 dy_2}
       {(|x-y_1| + |x-y_2|)^{\lambda}}.
\]
It was shown in \cite[Lemma 7]{KenigStein1999} that
the bilinear
fractional integral is bounded from $L^{p_1}\times L^{p_2}$
to $L^{q}$ when the indices satisfy certain conditions.
We refer to \cite{Benyi2015,Cao2019,Chaffee2015,Chen2011,Hoang2018,
Komori-Furuya2020,Li2017,LiSun2016,
Moen2014} for some recent advances on the study of
bilinear
fractional integrals.

Komori-Furuya
\cite{Komori2017}
gave a complete characterization of the indices
for which the multilinear fractional integral is bounded.
Here we cite a bilinear version.

\begin{Proposition}[\cite{Komori2017}]\label{prop:P1}
Suppose that $1\le p_1,p_2\le \infty$, $0<\lambda<2n$ and $1/p_1 + 1/p_2 =
1/q + (2n-\lambda)/n$. Then the norm estimate
\[
  \left\|\int_{\bbR^{2n}}
    \frac{f_1(y_1)f_2(y_2) dy_1 dy_2}
     {(|x-y_1| + |x-y_2|)^{\lambda}}
     \right\|_{L_x^q}
     \lesssim \|f_1\|_{L^{p_1}} \|f_2\|_{L^{p_2}}
\]
is true if and only if
\begin{enumerate}
\item either $1<p_1<\infty$ or $1<p_2<\infty$,
\item the index $q$ satisfies
   \begin{align*}
  \begin{cases}
\max\{p_1,p_2\}\le q <\infty,  & \mbox{if }  \min\{p_1,p_2\}=1,  \\
\min\{p_1,p_2\}< q <\infty,  & \mbox{if }  \max\{p_1,p_2\}=\infty,\\
0<1/q <1/p_1 + 1/p_2,  &  \mbox{if }  1< p_1, p_2<\infty \mbox{ and }
  1/p_1 + 1/p_2<  1,
 \\
0\le 1/q <1/p_1 + 1/p_2,  &  \mbox{if }  1< p_1, p_2<\infty
      \mbox{ and }
  1/p_1 + 1/p_2\ge   1.
\end{cases}
   \end{align*}
\end{enumerate}
\end{Proposition}

Kenig and Stein \cite[Remark 10]{KenigStein1999}
studied the multi-linear fractional integral of the following type,
\[
I(f_1,\ldots,f_k)(x)=  \int_{\bbR^{nk}}
    \frac{f_1(l_1)\ldots f_k(l_k) dy_1\ldots dy_k}
    {(|y_1|+\ldots + |y_k|)^{\lambda}},
\]
where
\[
l_i := l_i(y_1,\ldots,y_k,x)= \sum_{j=1}^k A_{i,j} y_j + A_{i,k+1}x
\]
are linear combinations of $y_1,\ldots,y_k,x\in\bbR^n$,
$A_{i,j}$ are $n\times n$ matrices such that
\begin{enumerate}
\item For each $1\le i\le k$, $A_{i,k+1}$ is invertible,
\item The $nk\times nk $ matrix $(A_{i,j})_{1\le i,j\le k}$
is invertible.
\end{enumerate}
They showed that when $1<p_i\le \infty$, $0<q<\infty$, $0<\lambda <kn$ and
\[
  \frac{1}{p_1} + \ldots + \frac{1}{p_k}
  = \frac{1}{q} + \frac{kn-\lambda}{n},
\]
$I(f_1,\ldots,f_k)$ is bounded from
$L^{p_1}\times \ldots \times L^{p_k}$
to $L^q$.

In this paper, we focus on the bilinear case.
Denote
\[
  A = \begin{pmatrix}
A_{11} & A_{12} \\
A_{21} & A_{22}
\end{pmatrix}
\quad\mbox{and}\quad
y = \begin{pmatrix}
y_1 \\
y_2
\end{pmatrix}.
\]
As in \cite{KenigStein1999}, we suppose that $A$ is invertible.
By a change of variable of the form
$y\rightarrow A^{-1}y$, we get
\begin{align*}
  I(f_1,f_2)(x)
&= \int_{\bbR^{2n}}
      \frac{f_1(A_{11}y_1\!+\!A_{12}y_2 \!+\!A_{13}x) f_2(A_{21}y_1\!+\!A_{22}y_2\!+\!A_{23}x)dy_1 dy_2}
      {(|y_1|+|y_2|)^{\lambda}}\\
  &\approx
    \int_{\bbR^{2n}}
      \frac{f_1(y_1+A_{13}x) f_2(y_2+A_{23}x)dy_1 dy_2}
      {(|y_1|+|y_2|)^{\lambda}}\\
   &=  \int_{\bbR^{2n}}
      \frac{f_1(y_1) f_2(y_2 )dy_1 dy_2}
      {(|A_{13}x - y_1|+|A_{23}x - y_2|)^{\lambda}} .
\end{align*}
It follows from  \cite[Remark 10]{KenigStein1999} that
$I(f_1,f_2)$ is bounded from
$L^{p_1}\times L^{p_2}$ to $L^q$ when both $A_{13}$ and $A_{23}$
are invertible and the indices satisfy certain conditions.

We show that for $I(f_1,f_2)$ to be bounded,
$A_{13}$ or $A_{23}$ might be singular.
Moreover,
we study
the bilinear fractional integral in more general settings.

Let $n_1, n_2$ and $m$ be positive integers,
$D_i$ be $n_i\times m$ matrix, $i=1,2$,
$0<p_1,p_2, q\le \infty$ and $0<\lambda<n_1+n_2$.
For   $f_i\in L^{p_i}(\bbR^{n_i})$, $i=1,2$
and $x\in\bbR^m$,
consider the bilinear fractional integral
\begin{equation}\label{eq:bilinear fractinoal}
  I_{\lambda,D}(f_1,f_2)(x)
   = \int_{\bbR^{n_1+n_2}} \frac{f_1(y_1) f_2(y_2) dy_1 dy_2}
   {(|D_1x-y_1|+|D_2x-y_2|)^{\lambda}}.
\end{equation}

We give necessary and sufficient conditions on
the matrices $D_1$ and $D_2$ and the indices
$p_1$, $p_2$, $q$ and $\lambda$ such that
$I_{\lambda,D}$ is bounded from
$L^{p_1}(\bbR^{n_1})\times L^{p_2}(\bbR^{n_2})$ to $L^q(\bbR^m)$.
Our result generalizes Proposition~\ref{prop:P1}.

\begin{Theorem} \label{thm:bilinear}
Suppose
that $n_1, n_2$ and $m$ are positive integers,
that $0<\lambda<n_1+n_2$,
that $0<p_1, p_2, q\le \infty$, and that $D_i$ are  $n_i\times m$ matrices, $i=1,2$. Denote $r_i=\rank(D_i)$. Then $I_{\lambda,D}$ is bounded from $L^{p_1}(\bbR^{n_1})
\times L^{p_2}(\bbR^{n_2})$ to $L^q(\bbR^m)$ if and only if
\begin{enumerate}
\item The rank of the $(n_1+n_2)\times m$ matrix $\binom{D_1}{D_2}$ is $m$;

\item The indices meet the homogeneity condition $ \lambda = n_1/p'_1 + n_2 /p'_2 + m/q$;

\item The index vector $\vec p=(p_1,p_2)$ satisfies $p_1, p_2\ge 1$,  $\#\{i:\, 1<p_i<\infty\}\ge 1$,
$p_1 < \infty$ when  $r_2 <m$, and $p_2<\infty$ when $r_1<m$;

\item The index  $q$ is finite when $\#\{i:\, 1<p_i<\infty\}=1$ or
             $1/p_1 + 1/p_2<1$.
     Moreover, $q$ satisfies
    one of the following conditions,

\begin{enumerate}

\item
when $r_1=r_2=m$,
\[
   \frac{1}{q} \le  \sum_{i:\, p_i>1} \frac{1}{p_i},
\]
where the equality is accessible only if (1) $\min\{p_1,p_2\}=1$
or (2) $1< p_1, p_2<\infty$, $n_1,n_2>m$ and $   1/p_1 + 1/p_2\ge   1$,

\item
when $r_1=0$ and $r_2=m$ (the case $r_1=m$ and $r_2=0$ is similar),
\[
  \begin{cases}
  q\ge p_1 , & \mbox{if } p_2=1 , \\
  q\ge p_2, & \mbox{if } 1<p_2<\infty, \mbox{ where the equality
   is accessible only   if }  n_2>m \mbox{ and }   \\
    & 1<p_1\le p'_2,\\
  \end{cases}
\]

\item
when $0<r_1<r_2=m$ or $0<r_2<r_1=m$,
\[
  \begin{cases}
 q\ge  \max\{p_1,p_2\} , & \mbox{if } \min\{p_1,p_2\}=1 , \\
 q> \min\{p_1,p_2\} , & \mbox{if } (p_1,r_2)=(\infty,m) \mbox{ or }
                               (p_2,r_1)=(\infty,m), \\
 q\ge  a_{p_1,p_2} , &\mbox{if } 1<p_1,p_2<\infty ,
  \end{cases}
\]
where $a_{p_1,p_2} = p_1$ when $r_1=m$,
 and $a_{p_1,p_2} =  p_2$ when $r_2=m$,

\item
when  $0<r_1, r_2<m$,
\[
  q \ge \max\{p_1,p_2\},
\]
where the equality  is accessible only if one of the following conditions is satisfied,
\begin{enumerate}
\item  $p_i=1$   for some  $i$ with $r_1+r_2>m$  or  $r_{3-i}<n_{3-i}$,

\item $1<p_1\ne  p_2 <\infty$,

\item $p_1=p_2$ and $r_1+r_2>m$,

\item $1<p_1=p_2\le 2$, $r_1+r_2=m$, $n_1>r_1$ and $n_2>r_2$.
\end{enumerate}

\end{enumerate}
\end{enumerate}
\end{Theorem}


The paper is organized as follows.
In Section 2, we
give a generalized version of the (linear) fractional integral.
In Sections 3 and 4, we give the proof of the necessity and
sufficiency of Theorem~\ref{thm:bilinear}, respectively.

\section{Preliminary Results}

In this section, we collect some preliminary results which are used in the proof of Theorem~\ref{thm:bilinear}.

\subsection{Operators That Commute with Translations}

It is well known that if a linear operator   commutes with translations
and is bounded from $L^p(\bbR^n)$ to $L^q(\bbR^n)$ for some $0<p,q<\infty$,
then $p\le q$. For our purpose, we need a sightly general version, which
can be proved similarly to \cite[Theorem 2.5.6]{Grafakos2014}.

\begin{Proposition}\label{prop:commute}
Suppose that $T$ is a non-zero linear operator which is bounded
from $L^p(\bbR^n)$ to $L^q(\bbR^m)$ for some $1\le p,q<\infty$.
For any $a\in\bbR$ and $1\le i\le n$, define  $f_a(y)=f(y_1, \ldots, y_{i-1}, y_i-a, y_{i+1}, \ldots, y_n)$.
Suppose that for some $1\le j\le m$,
\[
  T (f_a)(x) = (Tf)(x_1, \ldots, x_{j-1}, x_j-a, x_{j+1}, \ldots, x_n),\qquad \forall a\in\bbR.
\]
Then we have $p\le q$.
\end{Proposition}

\begin{proof} For any $f\in L^p$ and $a\in\bbR$, we have
\begin{align*}
\lim_{a\rightarrow\infty}
  \|T f_a + Tf \|_{L^q}
&=\lim_{a\rightarrow\infty}
  \|  (Tf)(x_1, \ldots, x_{j-1}, x_j-a, x_{j+1}, \ldots, x_n)
     - (Tf)(x)\|_{L^q} \\
&= 2^{1/q}\|Tf\|_{L^q}.
\end{align*}
On the other hand,
\begin{align*}
\lim_{a\rightarrow\infty}
  \|T f_a + Tf \|_{L^q}
&\le \lim_{a\rightarrow\infty}
  \|T\| \cdot\| f_a + f \|_{L^p}
= 2^{1/p}\|T\|\cdot \|f\|_{L^p}.
\end{align*}
Hence
\[
  2^{1/q}\|Tf\|_{L^q}\le 2^{1/p}\|T\|\cdot \|f\|_{L^p}.
\]
Therefore, $1/q \le 1/p$. That is, $q\ge p$.
\end{proof}

\subsection{Generalized Fractional Integrals}

For $f\in L^p(\bbR^n)$ and $0<\lambda<n$,
the  fractional integral, also known as the Reisz potential,  of $f$ is defined by
\[
  I_{\lambda}f(x) = \int_{\bbR^n} \frac{f(y) dy}{|x-y|^{\lambda}}.
\]
It was shown that $I_{\lambda}$ is bounded from $L^p$ to $L^q$
whenever $1<p<q<\infty$. We refer to  \cite[Theorem 1.2.3]{Grafakos2014m}
or
\cite[Chapter 5.1]{Stein1970}
for a proof.
In this paper, we consider a generalized version.

\begin{Theorem}\label{thm:linear}
Let $n,m$ be positive integers, $0<\lambda<n$, $0<p,q\le \infty$
and $D$ be an $n\times m$ matrix.
Then the norm estimate
\begin{equation}\label{eq:Reisz}
   \left\| \int_{\bbR^n} \frac{f(y) dy }{|Dx-y|^{\lambda}}
    \right\|_{L_x^q}
    \lesssim \|f\|_{L^p}
\end{equation}
is true if and only if $\rank(D)=m$, $1<p<q<\infty$ and
$\lambda = n/p'+m/q$.
\end{Theorem}

\begin{proof}
Sufficiency.
Since $\rank(D)=m$, we have $n\ge m$.
There are two cases.

(i)\,\, $n=m$.

In this case, $D$ is invertible. The conclusion follows by the boundedness
of the Riesz potential.

(ii)\,\, $n>m$.

In this case, there exist an $n\times n$ invertible matrix $P$
and an $m\times m$ invertible matrix $Q$ such that
$
  D = P \binom{Q}{0}$.
Since $|Dx - y| \approx  |P^{-1}Dx - P^{-1}y|$,
By replacing $f$ with $f(P^{-1}\cdot)$ in (\ref{eq:Reisz}),
we get that (\ref{eq:Reisz}) is equivalent to
\[
   \left\| \int_{\bbR^n} \frac{f(y) dy }{|P^{-1}Dx-y|^{\lambda}}
    \right\|_{L_x^q}
    \lesssim \|f\|_{L^p}.
\]
Denote $y = (y_1, y_2)$, where $y_1\in \bbR^{m}$ and $y_2 \in \bbR^{n-m}$.
By a change of variable of the form $x\rightarrow Q^{-1}x$,
we see that the above inequality is equivalent to
\[
    \left\| \int_{\bbR^n} \frac{f(y_1,y_2) dy_1 dy_2 }
    {(|x-y_1| + |y_2|)^{\lambda}}
    \right\|_{L_x^q}
    \lesssim \|f\|_{L^p}.
\]
By H\"older's inequality, we have
\[
  \int_{\bbR^n} \frac{f(y_1,y_2) dy_1 dy_2 }
    {(|x-y_1| + |y_2|)^{\lambda}}
    \lesssim
 \int_{\bbR^{m}} \frac{\|f(y_1,\cdot)\|_{L_{y_2}^{p}} dy_1}
    { |x-y_1|^{\lambda-(n-m)/p'}}.
\]
Now we get the conclusion as desired by the
boundedness of the Riesz potential.

Necessity.
First, we show that $\rank(D)=m$.
Assume on the contrary that $\rank(D)<m$.
Then there is some invertible $m\times m$ matrix $Q$ such that
the last column of $DQ$ is zero.
By a change of variable of the form $x\rightarrow Qx$, we get
\begin{align*}
\left\| \int_{\bbR^n} \frac{f(y)dy}{|Dx-y|^{\lambda}}\right\|_{L_x^q}
&=
  \left\| \int_{\bbR^n} \frac{f(y)dy}
    {|DQx-y|^{\lambda}}\right\|_{L_x^q}
    \\
&=
  \left\|  \mbox{a function of $(x^{(1)}, \ldots, x^{(m-1)})$}\right\|_{L_x^q}
    \\
&=\infty,
\end{align*}
which is a contradiction.

By replacing $f$ with $f(\cdot/a)$, we get the homogeneity condition $\lambda = n/p'+m/q$. It remains to show that
$1<p<q<\infty$.

As in the sufficiency part, we have
\begin{equation}\label{eq:e6}
  \left\| \int_{\bbR^{n}}
    \frac{f(y_1, y_2) dy_1 dy_2}{(|x-y_1| + |y_2|)^{\lambda}}
    \right\|_{L_x^q}
    \lesssim \|f\|_{L^p},
\end{equation}
where $y_1\in \bbR^m$ and $y_2\in \bbR^{n-m}$.
If $n=m$, then we delete the variable $y_2$.
%

First, we show that $q<\infty$.
If $q=\infty$,
then we see from (\ref{eq:e6}) that
for any $h\in L^1$,
\[
  \int_{\bbR^{n+m}}
    \frac{|f(y_1, y_2) h(x)| dx dy_1 dy_2}{(|x-y_1| + |y_2|)^{\lambda}}
    \lesssim \|f\|_{L^p}\|h\|_{L^1}.
\]
Set $h = (1/\delta^m) \chi^{}_{\{|x|\le \delta\}}$ and let
$\delta\rightarrow 0$. We see from Fatou's lemma that
\begin{equation}\label{eq:e:s1}
  \int_{\bbR^{n}}
    \frac{|f(y)  |   dy }{ | y |^{\lambda}}
    \lesssim \|f\|_{L^p},
\end{equation}
which is impossible since $ 1/| y |^{\lambda}\not \in L^{p'}$ when
$p\ge 1$ and $(L^p)^* = \{0\}$ when $0<p<1$. Hence $q<\infty$.

Next we show that $1<p<\infty$.
If $p\le 1$, then $\lambda = n/p'+m/q \le m/q$. By setting $f =
\chi^{}_{\{|y_1|+|y_2|\le 1\}}$ in (\ref{eq:e6}), we get
\[
    \left\|
    \frac{1}{(1+|x|)^{\lambda}}
    \right\|_{L_x^q} <\infty,
\]
which is impossible since $\lambda q \le m$.

If $p=\infty$, by setting $f\equiv 1$, we get
\[
  \int_{\bbR^{n}}
    \frac{f(y_1, y_2) dy_1 dy_2}{(|x-y_1| + |y_2|)^{\lambda}} = \infty.
\]
Hence $1<p<\infty$.

Finally,  we show that $q> p$.
If $n=m$, we see from $\lambda = n/p' + m/q<n$ that $q>p$. 
It remains to consider the case $n>m$.

Denote
\[
  Tf(x) = \int_{\bbR^{n}}
    \frac{f(y_1, y_2) dy_1 dy_2}{(|x-y_1| + |y_2|)^{\lambda}}.
\]
If $q<1$, we see from Minkowski's inequality that when
$f$ is non-negative, 
\begin{align*}
\|Tf\|_{L^q} 
\ge \int_{\bbR^{n}}
    \left\|\frac{f(y_1, y_2) }{(|x-y_1| + |y_2|)^{\lambda}}\right\|_{L_x^q}dy_1 dy_2
 \approx 
 \int_{\bbR^{n}}
        \frac{f(y_1, y_2) dy_1 dy_2 }{|y_2|^{n/p'}}.
\end{align*}
Take some $f(y_1,y_2)=f_1(y_1)f_2(y_2)$ with
$f_1, f_2\ge 0$, $f_2\in L^p$ and $f_1\in L^p\setminus L^1$,
we get a contradiction.
Hence $q\ge 1$.

For $z=(0,\ldots,0,a)\in\bbR^m$, we have
\begin{align*}
   (T f(\cdot-z, \cdot))(x)
  &=  \int_{\bbR^{n}}
    \frac{f(y_1-z, y_2) dy_1 dy_2}{(|x-y_1| + |y_2|)^{\lambda}}\\
  &=(Tf)(x-z).
\end{align*}
By Proposition~\ref{prop:commute}, we have $p\le q$.

If $p=q$, then $\lambda = (n-m)/p' + m$. Since $\lambda<n$, we have $m<n$.
It follows from     (\ref{eq:e6}) that for any $h\in L^{q'}$,
\begin{equation}\label{eq:e7}
 \left\|  \int_{\bbR^m} \frac{h(x)dx}{(|x-y_1| + |y_2|)^{\lambda}}
  \right\|_{L_y^{p'}}
   \lesssim \|h\|_{L^{q'}}.
\end{equation}
Set $h = \chi^{}_{\{|x|\le 1\}}$. For $\delta>0$ small enough
and $|y_1|, |y_2|\le \delta$, we have
\begin{align*}
\int_{\bbR^m} \frac{h(x)dx}{(|x-y_1| + |y_2|)^{\lambda}}
&\ge \int_{|x-y_1|\le |y_2|} \frac{h(x)dx}{(|x-y_1| + |y_2|)^{\lambda}} \\
&\gtrsim \frac{1}{|y_2|^{\lambda - m}}.
\end{align*}
Since $\lambda - m = (n-m)/p'$, we have
\[
  \left\|  \int_{\bbR^m} \frac{h(x)dx}{(|x-y_1| + |y_2|)^{\lambda}}
  \right\|_{L_y^{p'}} = \infty,
\]
which is a contradiction.
This completes the proof.
\end{proof}

Next we show that when $|x-y|$ is replaced by $|x|+|y|$ in the definition
of $I_{\lambda}$,
then even for $p=q$, the operator is bounded from $L^p$ to $L^q$.

\begin{Lemma}\label{Lm:L1}
Suppose that $0<p, q\le \infty$ and $\lambda>0$.
Then the inequality
\begin{equation}\label{eq:L1}
  \left\|\int_{\bbR^n}
   \frac{f(y)  dy   }{(|x|+|y|)^{\lambda}}
   \right\|_{L^q(\bbR^m)}
   \lesssim \|f\|_{L^{p}}, \qquad f\in L^p(\bbR^n),
\end{equation}
is true if and only if $\lambda = n/p'+m/q$ and
$1< p\le q<\infty$.

As a result, for $1\le p_1, p_2 \le \infty$ with $(p_1, p_2)\ne (1,1)$,
\begin{equation}\label{eq:L1a}
  \left |\int_{\bbR^{n_1+n_2}}
   \frac{f_1(y_1)  f_2(y_2) dy_1 dy_2   }
    {(|y_1|+|y_2|)^{n_1/p'_1+n_2/p'_2}}
   \right |
   \lesssim \|f_1\|_{L^{p_1}(\bbR^{n_1})}
   \|f_2\|_{L^{p_2}(\bbR^{n_2})}
\end{equation}
  if and only if $1<p_1, p_2<\infty$ and $1/p_1 + 1/p_2 \ge 1$.
\end{Lemma}

\begin{proof}
Denote
\[
  T_{\lambda}f(x) = \int_{\bbR^n}
   \frac{f(y)  dy   }{(|x|+|y|)^{\lambda}}.
\]
\textit{Necessity}.
As in the proof of Theorem~\ref{thm:linear},
by replacing $f$ with $f(\cdot/a)$, we get the homogeneity condition $\lambda = n/p'+m/q$. It remains to show that
$1< p\le q<\infty$.

If $q=\infty$, then for any $f\in L^p$ and $h\in L^1(\bbR^m)$,
\[
    \int_{\bbR^{m+n}}
   \frac{|f(y)h(x)|  dx dy  }{(|x|+|y|)^{\lambda}}
   \lesssim  \|f\|_{L^p}\|h\|_{L^1}.
\]
By setting $h = (1/\delta^m) \chi^{}_{\{|x|\le \delta\}}$ and letting
$\delta\rightarrow 0$, we  get (\ref{eq:e:s1}),
which is impossible. Hence $q<\infty$.

If $p\le 1$, then $\lambda = n/p' +m/q \le m/q$.
Set $f = \chi^{}_{\{|y|\le 1\}}$.
We have
\[
  |T_{\lambda}f(x)| \ge \frac{1}{(1+|x|)^{\lambda}}.
\]
Hence $T_{\lambda}f \not\in L^q$, which is a contradiction.

It remains to show that $p\le q$.
Set
\[
f(y)  = \frac{\chi^{}_{\{|y|\le 1/2\}}(y)}
     {|y|^{n/p} (\log 1/|y|)^{(1+\varepsilon)/p}},
\]
where $\varepsilon>0$ is a constant. We have $f\in L^p$.
For $\delta>0$ small enough and $|x|\le \delta$, we have
\begin{align*}
 \int_{\bbR^n}
   \frac{f(y)  dy }{(|x|+|y|)^{\lambda}}
   &\ge
\int_{|x|^2 \le |y|\le |x| }
   \frac{ dy }{(|x|+|y|)^{n/p'+m/q}
   |y|^{n/p} (\log 1/|y|)^{(1+\varepsilon)/p}
   } \\
&\gtrsim
   \frac{1}{|x|^{m/q} (\log 1/|x|)^{(1+\varepsilon)/p} }.
\end{align*}
If $q<p$, we can choose $\varepsilon>0$ small enough
such that $q (1+\varepsilon)/p <1$. Consequently,
\[
\left\|\int_{\bbR^n}
   \frac{f(y)  dy }{(|x|+|y|)^{\lambda}}
   \right\|_{L_x^q}
   =\infty,
\]
which contradicts the assumption.
Hence $p\le q$.

\textit{Sufficiency}.
Define   the operator $S$   by
\[
  Sf(x) = \int_{\bbR^n} \frac{f(y)dy}{(|x|+|y|)^{\lambda}},\quad x\in\bbR^m.
\]
Let $1<r,s <\infty$ be such that $n/r' + m/s = \lambda$. For any $f\in L^r$,
it follows from H\"older's inequality that
\[
    |Sf(x)|\le \frac{\|f\|_{L^r}}{|x|^{m/s}}.
\]
Hence
$\|Sf\|_{L^{s,\infty}} \lesssim \|f\|_{L^r}$. By the interpolation theorem,
 $S$ is bounded from $L^p$ to $L^q$ whenever $n/p' + m/q = \lambda$
 and $q\ge p$. Hence
(\ref{eq:L1}) is true. This completes the proof.
\end{proof}

\section{Proof of Theorem~\ref{thm:bilinear}: The Necessity}

In this section, we give a proof of the necessity part
in Theorem~\ref{thm:bilinear}.
We prove the conclusion in several steps.

(S1)\,\, We show that
$\rank\binom{D_1}{D_2} = m$.

Since $D_i$  are $n_i\times m$ matrices, $i=1,2$,
we have $\rank\binom{D_1}{D_2} \le m$.
Assume that
$\rank\binom{D_1}{D_2} < m$. Then there is some $m\times m$ invertible
matrix $P$ such that
the last column of $\binom{D_1}{D_2}P$ is zero.
By a change of variable of the form
$x\rightarrow Px$, we get
\begin{align*}
 \left\| I_{\lambda,D}(f_1,f_2)\right\|_{L^q}
   &= \left\| \int_{\bbR^{n_1+n_2}} \frac{f_1(y_1) f_2(y_2) dy_1 dy_2}
   {(|D_1x-y_1|+|D_2x-y_2|)^{\lambda}} \right\|_{L^q}\\
&\approx
  \left\| \int_{\bbR^{n_1+n_2}} \frac{f_1(y_1) f_2(y_2) dy_1 dy_2}
   {(|D_1Px-y_1|+|D_2Px-y_2|)^{\lambda}} \right\|_{L^q} \\
& = \left\| \mbox{ a function of $(x^{(1)}, \ldots, x^{(m-1)})$}   \right\|_{L^q} \\
&=\infty,
\end{align*}
which is a contradiction.

(S2)\,\, We show that the indices $p_1, p_2, q$ and $\lambda$ satisfy the homogeneity condition.

For $a>0$, set $f_{i,a} = f_i(\cdot/a)$, $i=1,2$. We have
\begin{align*}
I_{\lambda,D}(f_{1,a},f_{2,a})(x)
   &= \int_{\bbR^{n_1+n_2}} \frac{f_1(y_1/a) f_2(y_2/a) dy_1 dy_2}
   {(|D_1x-y_1|+|D_2x-y_2|)^{\lambda}}
     \\
&= \int_{\bbR^{n_1+n_2}} \frac{a^{n_1+n_2}f_1(y_1 ) f_2(y_2 ) dy_1 dy_2}
   {(|D_1x-ay_1|+|D_2x-ay_2|)^{\lambda}}
     \\
&= \int_{\bbR^{n_1+n_2}} \frac{a^{n_1+n_2-\lambda}f_1(y_1 ) f_2(y_2 ) dy_1 dy_2}
   {(|D_1x/a-y_1|+|D_2x/a-y_2|)^{\lambda}}
     \\
&= a^{n_1+n_2-\lambda}I_{\lambda,D}(f_1,f_2)(x/a).
\end{align*}
Since $I_{\lambda,D}$ is bounded, we have
\[
  \|a^{n_1+n_2-\lambda}I_{\lambda,D}(f_1,f_2)(x/a)\|_{L^q}
  \lesssim
  C
  \|f_{1,a}\|_{L^{p_1}} \|f_{2,a}\|_{L^{p_2}},
\]
where $C=
\| I_{\lambda,D}\|_{L^{p_1}\times L^{p_2}\rightarrow L^q}$.
Hence
\[
  a^{n_1+n_2-\lambda+m/q} \|I_{\lambda,D}(f_1,f_2)\|_{L^q}
  \lesssim
  a^{n_1/p_1+n_2/p_2}  C
  \|f_1\|_{L^{p_1}} \|f_2\|_{L^{p_2}}.
\]
Since $a>0$ is arbitrary, we have
$n_1+n_2-\lambda+m/q = n_1/p_1+n_2/p_2$.
Thus $ \lambda = n_1/p'_1 + n_2 /p'_2 + m/q$.

(S3)\,\, We show that $p_i\ge 1$, $i=1,2$.

Without loss of generality, assume that $p_1<1$.  Then
$\lambda <n_2 - n_2/p_2 + m/q$.
Hence
there is some $\alpha>n_2/p_2$ such that
$\lambda+\alpha - n_2 <m/q$.
Set $f_1(y_1) = \chi^{}_{\{|y_1|\le 1\}}(y_1)$
and $f_2(y_2) = \frac{1}{(1+|y_2|)^{\alpha}}$.
Then we have $f_i\in L^{p_i}$. However,
\begin{align*}
  I_{\lambda,D}(f_1,f_2)(x)
   &= \int_{\bbR^{n_1+n_2}} \frac{f_1(y_1) f_2(y_2) dy_1 dy_2}
   {(|D_1x-y_1|+|D_2x-y_2|)^{\lambda}} \\
&\gtrsim
    \int_{\bbR^{n_2}} \frac{  f_2(y_2)   dy_2}
      {(1+|x|+|y_2|)^{\lambda}}\\
&\ge
    \int_{\bbR^{n_2}} \frac{   dy_2}
      {(1+|x|+|y_2|)^{\lambda+\alpha }}\\
&\approx
   \frac{1}{(1+|x|)^{\lambda+\alpha - n_2}}
\end{align*}
which is not in $L^q$ when $q<\infty$.

When $q=\infty$ and $p_2<\infty$, the above choices
of $f_1$ and $f_2$ lead to
\begin{align*}
  I_{\lambda,D}(f_1,f_2)(x)
    &\gtrsim
    \int_{\bbR^{n_2}} \frac{   dy_2}
      {(1+|x|+|y_2|)^{\lambda+\alpha }}
 =\infty.
\end{align*}
When $p_2=q=\infty$, we have $\lambda<n_2$.
Set $f_1 = \chi^{}_{\{|y_1|\le 1\}}$
and $f_2 \equiv 1$.
We get
\begin{align*}
  I_{\lambda,D}(f_1,f_2)(x)
   &= \int_{\bbR^{n_1+n_2}} \frac{f_1(y_1) f_2(y_2) dy_1 dy_2}
   {(|D_1x-y_1|+|D_2x-y_2|)^{\lambda}} \\
&\gtrsim
    \int_{\bbR^{n_2}} \frac{    dy_2}
      {(1+|x|+|y_2|)^{\lambda}}\\
&=\infty.
\end{align*}
Again, we get a contradiction.
Hence $p_1, p_2\ge 1$.

(S4)\,\, We show that
either $1<p_1<\infty$ or $1<p_2<\infty$.

Assume on the contrary  that $p_i=1$ or the infinity.
There are three cases.

(i)\,\, $p_1=p_2=1$.

In this case, $\lambda = m/q$. Hence $q<\infty$.
Set $f_i = \chi^{}_{\{|y_i|\le 1\}}$, $i=1,2$. We have
\begin{align*}
  I_{\lambda,D}(f_1,f_2)(x)
   &= \int_{\bbR^{n_1+n_2}} \frac{f_1(y_1) f_2(y_2) dy_1 dy_2}
   {(|D_1x-y_1|+|D_2x-y_2|)^{\lambda}} \\
&\gtrsim
    \frac{1}{(1+|x|)^{\lambda}}.
\end{align*}
Therefore, $I_{\lambda,D}(f_1,f_2)\not\in L^q$, which is a contradiction.

(ii)\,\, $p_1=p_2=\infty$.

In this case, $\lambda = n_1/p'_1 + n_2/p'_2 + m/q \ge n_1+n_2$, which contradicts the hypothesis.

(iii)\,\, One of $p_1$ and $p_2$ equals $1$,
and the other  equals   $\infty$.

Without loss of generality, assume that  $p_1=1$ and $p_2=\infty$.
Then we have $\lambda = n_2+m/q$. Set $f_2 \equiv 1$.

If $q=\infty$, then
\begin{align*}
  I_{\lambda,D}(f_1,f_2)(x)
   &= \int_{\bbR^{n_1+n_2}} \frac{f_1(y_1)  dy_1 dy_2}
   {(|D_1x-y_1|+|D_2x-y_2|)^{n_2}}
 =\infty.
\end{align*}

If $q<\infty$, we have
\begin{align*}
  I_{\lambda,D}(f_1,f_2)(x)
   &= \int_{\bbR^{n_1+n_2}} \frac{f_1(y_1)  dy_1 dy_2}
   {(|D_1x-y_1|+|D_2x-y_2|)^{\lambda}}\\
 &\approx\int_{\bbR^{n_1}}\frac{f_1(y_1)  dy_1  }
   {|D_1x-y_1|^{m/q}}.
\end{align*}
Hence
\[
  \left\|\int_{\bbR^{n_1}}\frac{f_1(y_1)  dy_1  }
   {|D_1x-y_1|^{m/q}}
   \right\|_{L^q} \lesssim \|f_1\|_1,
\]
which is impossible, thanks to Theorem~\ref{thm:linear}.

In both cases, we get contradictions. Hence there is some $i$
such that $1<p_i<\infty$.

(S5)\,\, We show that $p_1<\infty$ when $r_2<m$,  and
$p_2<\infty$ when $r_1<m$.

Assume that  $r_2<m$ and $p_1=\infty$.
Set $f_1 \equiv 1$. We see from $\| I_{\lambda,D}(f_1,f_2)\|_{L^q}
\lesssim \|f_1\|_{L^{p_1}} \|f_2\|_{L^{p_2}}$ that
\[
  \left\| \int_{\bbR^{n_2}}
     \frac{f_2(y_2) dy_2}{|Dx-y_2|^{n_2/p'_2 + m/q}}
     \right\|_{L_x^q}
  \lesssim \|f_2\|_{L^{p_2}},
\]
which contradicts Theorem~\ref{thm:linear} since $r_2=\rank(D_2)<m$.

Similarly we get that $p_2<\infty$ when $r_1<m$.

(S6)\,\, We show that $q<\infty$ when one of the three conditions is satisfied, $\min\{p_1,p_2\}=1$,
$\max\{p_1,p_2\}=\infty$
or $1/p_1 + 1/p_2 <1$.

Assume on the contrary that $q=\infty$. Then $\lambda = n_1/p'_1 + n_2/p'_2$. For any $h\in L^1$,
\[
  \int_{\bbR^{n_1+n_2+m}}
    \frac{|f_1(y_1) f_2(y_2) h(x)|dx dy_1 dy_2 }
    {(|D_1x-y_1|+|D_2x-y_2|)^{\lambda}}
   \lesssim \|f_1\|_{L^{p_1}} \|f_2\|_{L^{p_2}}\|h\|_{L^1}.
\]
By setting $h = (1/\delta^m) \chi^{}_{\{|x|\le \delta\}}$
and letting $\delta\rightarrow 0$, we get
\[
  \int_{\bbR^{n_1+n_2 }}
    \frac{|f_1(y_1) f_2(y_2) | dy_1 dy_2 }
    {(|y_1|+|y_2|)^{n_1/p'_1 + n_2/p'_2}}
   \lesssim \|f_1\|_{L^{p_1}} \|f_2\|_{L^{p_2}} .
\]
By Lemma~\ref{Lm:L1}, $1<p_1, p_2<\infty$ and $1/p_1 + 1/p_2 \ge 1$,
 which contradicts the assumption. Hence $q<\infty$.

In the following subsections, we prove the rest part
 in several cases.

\subsection{The case $r_1 = r_2 = m$}\label{subsec:r1r2=m}

Since $\rank(D_i) = m$
and $ D_i $ are $n_i\times m$ matrices,   we have
$n_1, n_2\ge m$. There are $n_i\times n_i$ invertible matrices
$P_i$, $i=1,2$, such that
\[
  P_1 D_1 = \begin{pmatrix}
  I_{m}\\
  0_{(n_1-m)\times m}
  \end{pmatrix}
  \quad
  \mathrm{and}
  \quad
  P_2 D_2 = \begin{pmatrix}
  I_{m}\\
  0_{(n_2-m)\times m}
  \end{pmatrix},
\]
where $I_m$ stands for the $m\times m$ identity matrix.
Note  that
\begin{align*}
 |D_1x-y_1| + |D_2x-y_2|
&\approx |P_1D_1x-Py_1| + |P_2D_2x-P_2y_2|  .
\end{align*}
We have
\begin{align*}
& \int_{\bbR^{n_1+n_2}} \frac{f_1(y_1) f_2(y_2) dy_1 dy_2}
   {(|D_1x-y_1|+|D_2x-y_2|)^{\lambda}} \\
&\approx  \int_{\bbR^{n_1+n_2}} \frac{f_1(y_1) f_2(y_2) dy_1 dy_2}
   {(|P_1D_1x-P_1y_1|+|P_2D_2x-P_2y_2|)^{\lambda}}  .
\end{align*}
By replacing $f_1(P_1\cdot)$ and $f_2(P_2\cdot)$ for
$f_1$ and $f_2$, respectively, we get that
the boundedness of $I_{\lambda,D}$ is equivalent to
\[
  \left\| \int_{\bbR^{n_1+n_2}} \frac{f_1(y_1) f_2(y_2) dy_1 dy_2}
   {(|P_1D_1x-y_1|+|P_2D_2x-y_2|)^{\lambda}} \right\|_{L^q}
 \lesssim \|f_1\|_{L^{p_1}} \|f_2\|_{L^{p_2}}.
\]
Denote $y_i = (y_{i1}, y_{i2})$, where $y_{i1}\in \bbR^m$
and $y_{i2} = \bbR^{n_i-m}$, $i=1,2$. The above inequality becomes
\begin{equation}\label{eq:bi:bounded}
  \left\| \int_{\bbR^{n_1+n_2}} \frac{f_1(y_{11},y_{12}) f_2(y_{21}, y_{22}) dy_1 dy_2}
   {(| x-y_{11}|+| x-y_{21}| + |y_{12}| + |y_{22}|)^{\lambda}} \right\|_{L^q}
 \lesssim \|f_1\|_{L^{p_1}} \|f_2\|_{L^{p_2}}.
\end{equation}

In the followings we prove the conclusion in several steps.

(i)\, We show that $q\ge \max\{p_1,p_2\}$ when
$ \min\{p_1,p_2\}=1$.

Without loss of generality, we assume that $p_1=1$. In this case, $1<p_2<\infty$
and $\lambda = n_2/p'_2 + m/q$.

Setting $f_1 = (1/\delta^{n_1})
\chi^{}_{\{|y_1|\le \delta\}}$
and letting $\delta\rightarrow 0$ in (\ref{eq:bi:bounded}), we see from Fatou's lemma
that
\begin{equation}\label{eq:e2}
  \left\| \int_{\bbR^{n_2}} \frac{ f_2(y_2) dy_2}
   {(| x|+| y_2|)^{\lambda}} \right\|_{L^q}
 \lesssim   \|f_2\|_{L^{p_2}},\qquad \forall f_2\in L^{p_2}.
\end{equation}
By Lemma~\ref{Lm:L1}, $q\ge p_2$.

(ii)\,\, We show that $q>\min\{p_1,p_2\}$ when
$\max\{p_1,p_2\}=\infty$.

Without loss of generality, assume that
$p_1=\infty$. Then $1<p_2<\infty$ and $\lambda = n_1 + n_2/p'_2 + m/q$.
Setting $f_1 \equiv 1$ in (\ref{eq:bi:bounded}), we have
\[
    \left\| \int_{\bbR^{n_2}} \frac{  f_2(y_{21}, y_{22})   dy_2}
   {( | x-y_{21}| +   |y_{22}|)^{n_2/p'_2+m/q}} \right\|_{L^q}
 \lesssim  \|f_2\|_{L^{p_2}},\qquad \forall f_2\in L^{p_2}.
\]
Now we see from Theorem~\ref{thm:linear} that
$p_2<q<\infty$.

In the followings we assume that $1<p_1, p_2<\infty$.

(iii)\, We prove that $1/q \le 1/p_1 +  1/p_2$.

Let
\[
   f_i(y_i) = \frac{\chi^{}_{\{|y_i|\le 1/2\}}(y_i)}
    {|y_i|^{n_i/p_i} (\log 1/|y_i|)^{(1+\varepsilon)/p_i}},
\]
where $\varepsilon>0$ is a constant.
We have $f_i\in L^{p_i}$, $i=1,2$. When $|x|$ is small enough, we have
\begin{align*}
&\int_{\bbR^{n_1+n_2}} \frac{f_1(y_{11},y_{12}) f_2(y_{21}, y_{22}) dy_1 dy_2}
   {(| x-y_{11}|+| x-y_{21}| + |y_{12}| + |y_{22}|)^{\lambda}} \\
&\ge
\int_{\substack{|x|^2 \le |y_i|\le |x|\\
 1\le i \le 2}} \frac{f_1(y_{11},y_{12}) f_2(y_{21}, y_{22}) dy_1 dy_2}
   {(| x-y_{11}|+| x-y_{21}| + |y_{12}| + |y_{22}|)^{\lambda}}\\
&\gtrsim
   \frac{1}{|x|^{m/q} (\log 1/|x|)^{
   (1+\varepsilon)(1/p_1+1/p_2)} }.
\end{align*}
If $1/q > 1/p_1+ 1/p_2$, then there is some $\varepsilon>0$ such that
$(1+\varepsilon)(1/p_1+1/p_2)q < 1$.
Consequently,
\[
\left\|
\int_{\bbR^{n_1+n_2}} \frac{f_1(y_{11},y_{12}) f_2(y_{21}, y_{22}) dy_1 dy_2}
   {(| x-y_{11}|+| x-y_{21}| + |y_{12}| + |y_{22}|)^{\lambda}} \\
\right\|_{L^q} = \infty,
\]
which contradicts (\ref{eq:bi:bounded}).
Hence  $1/q \le 1/p_1+ 1/p_2$.

(iv)\, We prove that $1/q < 1/p_1+ 1/p_2$ when $\min\{n_1,n_2\}=m$ or
$1/p_1 + 1/p_2 <1$.

Assume on the contrary that $1/q = 1/p_1+ 1/p_2$.
There are three cases.

(a)\,\, $n_1=n_2=m$.

In this case,  $\lambda  = n_1/p'_1+n_2/p'_2 +m/q=n_1+n_2$,
which contradicts the hypothesis.

(b)\,\, $\min\{n_1,n_2\}=m$ and $\max\{n_1,n_2\}>m$.

Without loss of generality, we assume that $n_1=m$ and $n_2>m$.
Denote $y_2 = (y_{21}, y_{22})$, where
$y_{21}\in\bbR^m$ and $y_{22}\in \bbR^{n_2-m}$.
Let
\[
  f_1(y_1)=\chi^{}_{\{|y_1|\le 1\}}(y_1)
  \quad\mathrm{and}\quad
   f_2(y_2) = \frac{\chi^{}_{\{|y_2|\le 1/2\}}(y_2)}
      {|y_{22}|^{(n_2-m)/p_2} (\log 1/|y_{22}| )^{(1+\varepsilon)/p_2}},
\]
where $\varepsilon>0$ is a constant such that
$(1+\varepsilon)/p_2<1$.

For $\delta>0$ small enough and $|x|\le \delta$,
we have
\begin{align*}
&\int_{\bbR^{m+ n_2}}
  \frac{f_1(y_1) f_2(y_2)   dy_1 dy_2}
   {(| x-y_1|+| x-y_{21}| +   |y_{22}|)^{\lambda}}
  \\
&\ge
\int_{\substack{|y_1-x|\le |y_{22}| \\    |y_{21}-x|\le |y_{22}|
  \\
  |y_{22}|\le \delta}}
  \frac{   dy_1 dy_{21} }
   {(| x-y_1|+| x-y_{21}| +   |y_{22}|)^{\lambda}
   }\\
&\qquad \times
   \frac{dy_{22}}{|y_{22}|^{(n_2-m)/p_2} (\log(1/|y_{22}|))^{(1+\varepsilon)/p_2}}
  \\
&\gtrsim
   \int_{|y_{22}|\le \delta}
     \frac{dy_{22}}{|y_{22}|^{n_2-m}
           (\log(1/|y_{22}|))^{(1+\varepsilon)/p_2}}\\
&=\infty.
\end{align*}
Hence
$\|I_{\lambda,D}(f_1,f_2)\|_{L^q}
  = \infty$,
which is a contradiction.

 (c)\,\, $\min\{n_1,n_2\}>m$ and $1/p_1 + 1/p_2 <1$.

In this case, $q>1$.
we see from (\ref{eq:bi:bounded}) that for any $h\in L^{q'}$,
\begin{align*}
\left|\int_{\bbR^{n_1+ n_2+m}}
  \frac{f_1(y_{11},y_{12}) f_2(y_{21}, y_{22}) h(x)dx dy_1 dy_2}
   {(| x-y_{11}|+| x-y_{21}| + |y_{12}|+  |y_{22}|)^{\lambda}}\right|
 \lesssim \|f_1\|_{L^{p_1}}   \|f_2\|_{L^{p_2}}\|h\|_{L^{q'}}.
\end{align*}
Hence
\begin{align}
\left\|\int_{\bbR^{m+n_1}}
  \frac{f_1(y_{11},y_{12})   h(x)dx dy_1 }
   {(| x-y_{11}|+| x-y_{21}| + |y_{12}|  + |y_{22}|)^{\lambda}}
   \right\|_{L_{y_2}^{p'_2}}
 \lesssim \|f_1\|_{L^{p_1}}  \|h\|_{L^{q'}}. \label{eq:e11}
\end{align}
Take
\[
  h(x)=\chi^{}_{\{|x|\le 1\}}(x)
  \quad\mathrm{and}\quad
   f_1(y_1) = \frac{\chi^{}_{\{|y_1|\le 1/2\}}(y_1)}
      {|y_{12}|^{(n_1-m)/p_1} (\log 1/|y_{12}| )^{(1+\varepsilon)/p_1}}.
\]
For $\delta>0$ small enough and $|y_2|\le \delta$, we have
\begin{align*}
& \int_{\bbR^{m+n_1}}
  \frac{f_1(y_{11},y_{12})   h(x)dx dy_1 }
   {(| x-y_{11}|+| x-y_{21}| +|y_{12}|  + |y_{22}|)^{\lambda}}\\
  &\gtrsim
 \int_{\substack{|x-y_{21}|\le |y_{22}| \\
 |y_{11}-y_{21}|\le |y_{22}| \\
 |y_{22}|^2\le |y_{12}|\le |y_{22}|
 }}
  \frac{ dx dy_1 }
   {(| x-y_{11}|+| x-y_{21}|  +|y_{12}|   + |y_{22}|)^{\lambda}} \\
&\qquad \times    \frac{1}
{|y_{12}|^{(n_1-m)/p_1} (\log(1/|y_{12}|))^{(1+\varepsilon)/p_1}}\\
  &\gtrsim
     \frac{1}{|y_{22}|^{(n_2-m)/p'_2}
     (\log(1/|y_{22}|))^{(1+\varepsilon)/p_1}} .
\end{align*}
Since $1/p_1 < 1/p'_2$,
we can choose
 $\varepsilon>0$ small enough
such that
$p'_2(1+\varepsilon)/p_1<1$. Thus
\begin{align*}
\left\| \frac{\chi_{\{|y_2|\le \delta\}}(y_2)}{|y_{22}|^{(n_2-m)/p'_2}
     (\log(1/|y_{22}|))^{(1+\varepsilon)/p_1}}\right\|_{L_{y_2}^{p'_2}}
     =\infty,
\end{align*}
which contradicts (\ref{eq:e11}).

\subsection{The case $r_1= 0$ and $r_2  = m$}
\label{subsec:r1=0}

In this case, $I_{\lambda,D}$ is bounded if and only if
\begin{equation}\label{eq:e16}
 \left\| \int_{\bbR^{n_1+n_2}}
   \frac{f_1(y_1) f_2(y_2)dy_1 dy_2}{(|y_1| + | x-y_{21}| + |y_{22}|)^{\lambda}}
   \right\|_{L_x^q}
   \lesssim \|f_1\|_{L^{p_1}} \|f_2\|_{L^{p_2}}.
\end{equation}

First, we show that $q\ge 1$. 
If $q<1$, 
take some non-negative functions $f_i\in L^{p_i}$, $i=1,2$.
We see from Minkowski's inequality that 
\begin{align*}
& \left\| \int_{\bbR^{n_1+n_2}}
   \frac{f_1(y_1) f_2(y_2)dy_1 dy_2}{(|y_1| + | x-y_{21}| + |y_{22}|)^{\lambda}}
   \right\|_{L_x^q} \\
& \ge
 \int_{\bbR^{n_1+n_2}}
   \left\|\frac{f_1(y_1) f_2(y_2)}{(|y_1| + | x-y_{21}| + |y_{22}|)^{\lambda}}\right\|_{L_x^q} dy_1 dy_2
     \\
&\approx 
 \int_{\bbR^{n_1+n_2}}
    \frac{f_1(y_1) f_2(y_2)}{(|y_1|   + |y_{22}|)^{\lambda-m/q}}  dy_1 dy_2.
\end{align*}
By (\ref{eq:e16}), we get 
\[
  \left|\int_{\bbR^{n_1+n_2}}
    \frac{f_1(y_1) f_2(y_2)}{(|y_1|   + |y_{22}|)^{\lambda-m/q}}  dy_1 dy_2\right|
  \lesssim \|f_1\|_{L^{p_1}} \|f_2\|_{L^{p_2}},  
\]
which is impossible when $p_2>1$.

For the case $p_2=1$, by setting $f_2= (1/\delta^{n_2}) 
\chi^{}_{\{|y_2|\le \delta\}}$ 
and letting $\delta\rightarrow 0$, we see 
from Fatou's lemma that 
\[
  \left|\int_{\bbR^{n_1}}
    \frac{f_1(y_1) }{|y_1|^{\lambda-m/q}}  dy_1\right|
  \lesssim \|f_1\|_{L^{p_1}},\qquad  \forall f\in L^{p_1},
\]
which is impossible since $ 1/|y_1|^{\lambda-m/q}\not\in L^{p'_1}$. Hence $q\ge 1$.

Next, we show that $q\ge p_2$. For $z=(0,\ldots,0,a)\in\bbR^m$,
where only the last coordinate is nonzero,
set $f_{2,z}(y_{21},y_{22}) = f_2(y_{21}-z, y_{22})$. We have
\[
  I_{\lambda,D}(f_1,f_{2,z})(x)
  =I_{\lambda,D}(f_1,f_2)(x-z).
\]
By Proposition~\ref{prop:commute},  $q\ge p_2$.

Now we see from (\ref{eq:e16}) that
 $I_{\lambda,D}$ is bounded if and only if for any $f_1\in L^{p_1}$, $f_2\in L^{p_2}$ and $h\in L^{q'}$,
\begin{equation}\label{eq:e17}
 \left| \int_{\bbR^{n_1+n_2+m}}
   \frac{f_1(y_1) f_2(y_2) h(x)  dx dy_1 dy_2}{(|y_1| + | x-y_{21}| + |y_{22}|)^{\lambda}}\right|
   \lesssim \|f_1\|_{L^{p_1}} \|f_2\|_{L^{p_2}} \|h\|_{L^{q'}}.
\end{equation}
There are three  cases.

(i)\,\, $p_2=1$.

In this case, $1<p_1<\infty$.
Set $f_2 = (1/\delta^{n_2}) \chi^{}_{\{|y_2|\le \delta\}}$.
By letting $\delta\rightarrow 0$, we see from
(\ref{eq:e17}) that
\[
   \left| \int_{\bbR^{n_1+m}}
   \frac{f_1(y_1)   h(x)  dx dy_1 }{(|y_1| + | x |  )^{\lambda}}\right|
   \lesssim \|f_1\|_{L^{p_1}}   \|h\|_{L^{q'}}.
\]
By Lemma~\ref{Lm:L1}, $q \ge p_1$.

(ii)\,\, $p_1=1$ or $p_1=\infty$.

In this case, $1<p_2<\infty$.
If $p_1=1$, by setting $f_1 = (1/\delta^{n_1}) \chi^{}_{\{|y_1|\le \delta\}}$ and letting $\delta\rightarrow 0$, we see from
(\ref{eq:e17}) that
\begin{equation}\label{eq:e18}
   \left| \int_{\bbR^{n_2+m}}
   \frac{f_2(y_2)   h(x)  dx dy_2 }{(|x-y_{21}| + | y_{22} |
      )^{n_2/p'_2 + m/q}}\right|
   \lesssim \|f_2\|_{L^{p_2}}   \|h\|_{L^{q'}}.
\end{equation}
Now we see from Theorem~\ref{thm:linear} that
$q>p_2$.

If $p_1=\infty$, by setting $f_1 \equiv 1$, we also have (\ref{eq:e18}).
Hence $q > p_2$.

(iii)\, $1<p_1, p_2<\infty$.

First,  we prove that $q>p_2$ when $n_2=m$.
Assume on the contrary that $q=p_2$. Then $\lambda = n_1/p'_1 + m$.
In this case, (\ref{eq:e16})
turns out to be
\begin{equation}\label{eq:e19}
 \left\| \int_{\bbR^{n_1+m}}
   \frac{f_1(y_1) f_2(y_2)dy_1 dy_2}{(|y_1| + | x-y_2| )^{\lambda}}
   \right\|_{L_x^q}
   \lesssim \|f_1\|_{L^{p_1}} \|f_2\|_{L^{p_2}}.
\end{equation}
Set
\[
     f_1 (y_1) = \frac{\chi^{}_{\{|y_1|\le 1/2\}}(y_1)}
       {|y_1|^{n_1/p_1} (\log 1/|y_1|)^{(1+\varepsilon)/p_1}}
     \quad
     \mbox{and}
     \quad
     f_2  = \chi^{}_{\{|y_2|\le 1\}},
\]
where $\varepsilon>0$ satisfies $1+\varepsilon < p_1$.
We have $f_1\in L^{p_1}$. For $\delta>0$ small enough
and $|x|, |y_2|\le \delta$, we have
\begin{align*}
&
\int_{\bbR^{n_1+m}}
   \frac{f_1(y_1) f_2(y_2)dy_1 dy_2}{(|y_1| + | x-y_2| )^{\lambda}}
  \\
&\ge
   \int_{\substack{|x-y_2|^2 \le |y_1|\le |x-y_2| \\
     |y_2|\le \delta}}
   \frac{ dy_1 dy_2}{(|y_1| + | x-y_2| )^{n_1/p'_1+m}
   |y_1|^{n_1/p_1} (\log 1/|y_1|)^{(1+\varepsilon)/p_1}
   } \\
&\gtrsim
   \int_{ |y_2|\le \delta}
   \frac{  dy_2}{  | x-y_2|^m
    (\log 1/|x-y_2|)^{(1+\varepsilon)/p_1}
   } \\
  &=\infty,
\end{align*}
which contradicts (\ref{eq:e19}). Hence $q>p_2$.

Next we prove that
$q>p_2$ when $n_2>m$ and $1/p_1 + 1/p_2 <1$.

Again, assume that $q=p_2$.
Set
\begin{align*}
  f_1(y_1) &= \frac{\chi^{}_{\{|y_1|\le 1/2\}}(y_1)}
     {|y_1|^{n_1/p_1} (\log 1/|y_1|)^{(1+\varepsilon)/p_1}},
  \\
  f_2(y_2) &= \frac{\chi^{}_{\{|y_2|\le 1/2\}}(y_1)}
  {|y_{22}|^{(n_2-m)/p_2} (\log 1/|y_{22}|)^{(1+\varepsilon)/p_2}},
\end{align*}
where
$\varepsilon>0$ satisfying
$(1+\varepsilon)/p_1 +(1+\varepsilon)/p_2<1$.
Then we have $f_i\in L^{p_i}$.
Moreover, for $\delta$ small enough,
\begin{align*}
&\int_{\bbR^{n_1+n_2}}
   \frac{f_1(y_1) f_2(y_2) dy_1 dy_2}{(|y_1| + | x-y_{21}| + |y_{22}|)^{\lambda}}
   \\
&\ge
  \int_{\substack{|y_{22}|^2\le |y_1|\le |y_{22}| \\
    |x-y_{21}|\le |y_{22}| \\
    |y_{22}|\le \delta }}
    \frac{  dy_1 dy_{21}}
    {(|y_1| + | x-y_{21}| + |y_{22}|)^{\lambda}
       }
   \\
 &\qquad \times \frac{dy_{22}}{|y_1|^{n_1/p_1} (\log 1/|y_1|)^{(1+\varepsilon)/p_1}
    |y_{22}|^{(n_2-m)/p_2} (\log 1/|y_{22}|)^{(1+\varepsilon)/p_2}}
   \\
&\gtrsim
    \int_{  |y_{22}|\le \delta }
    \frac{dy_{22}}{
    |y_{22}|^{ n_2-m } (\log 1/|y_{22}|)^{(1+\varepsilon)(1/p_1+1/p_2)}}
   \\
&=\infty.
\end{align*}
Again, we get a contradiction.

\subsection{The case $0<r_1<r_2=m$ or $0<r_2<r_1=m$}

We consider only the case $0<r_1<r_2=m$. The other case can be proved similarly.

Since $\rank(D_1)=r_1<m$, there exist invertible matrices $P_1$ and  $Q$
such that
\[
  P_1 D_1 Q=   \begin{pmatrix}
   I_{r_1} & 0\\
   0 & 0
   \end{pmatrix},
\]
where $I_r$ is the $r\times r$ identity matrix.
Since $\rank(D_2Q)=\rank(D_2)=m$, there exists an $n_2\times n_2$ invertible
matrix $P_2$ such that
\[
  P_2 D_2 Q = \begin{pmatrix}
   I_m  \\
   0
   \end{pmatrix},
\]
where $I_m$ is the $m\times m$ identity matrix.

Note that
\[
  |D_1x-y_1| + |D_2x-y_2|
  \approx
  |P_1D_1x - P_1 y_1|
  + |P_2D_2 x - P_2 y_2|.
\]
By a change of variable of the form $x\rightarrow Qx$ and
replacing $f_1(P_1\cdot)$ and $f_2(P_2\cdot)$ for
$f_1$ and $f_2$ respectively, we know that
\[
 \left\|
  \int_{\bbR^{n_1+n_2}}
    \frac{f_1(y_1) f_2(y_2)dy_1 dy_2 }
    {(|D_1x-y_1| + |D_2x-y_2|)^{\lambda}}
    \right\|_{L_x^q}
   \lesssim \|f_1\|_{L^{p_1}} \|f_2\|_{L^{p_2}}
\]
is equivalent to
\begin{equation}\label{eq:1e16}
 \left\| \int_{\bbR^{n_1+n_2}}
   \frac{f_1(y_1) f_2(y_2)dy_1 dy_2}{(|x_1-y_{11}| + |y_{12}| + | x-y_{21}| + |y_{22}|)^{\lambda}}
   \right\|_{L_x^q}
   \lesssim \|f_1\|_{L^{p_1}} \|f_2\|_{L^{p_2}},
\end{equation}
where
$x=(x_1,x_2)$, $x_1\in\bbR^{r_1}$,
$x_2 \in\bbR^{m-r_1}$,
$y_{i1}\in\bbR^{r_i}$ and $y_{i2}\in \bbR^{n_i-r_i}$.
If $n_1=r_1$ or $n_2=r_2$, then we delete the variable $y_{12}$
or $y_{22}$ accordingly.

The same arguments as in the case $(r_1,r_2)= (0,m)$ show
that   $q\ge p_2$.
It remains to show that
$q\ge  p_1 $ when $p_2=1$,
and $q>p_2$ when $p_1=\infty$.

Note that (\ref{eq:1e16}) is equivalent to
\begin{equation}\label{eq:1e17}
 \left| \int_{\bbR^{n_1+n_2+m}}
   \frac{f_1(y_1) f_2(y_2) h(x)  dx dy_1 dy_2}{(|x_1-y_{11}| + |y_{12}| + | x-y_{21}| + |y_{22}|)^{\lambda}}\right|
   \lesssim \|f_1\|_{L^{p_1}} \|f_2\|_{L^{p_2}} \|h\|_{L^{q'}}.
\end{equation}

There are two cases.

(i)\,\, $p_2=1$.

In this case, $1<p_1<\infty$.
Set $f_2 = (1/\delta^{n_2}) \chi^{}_{\{|y_2|\le \delta\}}$.
By letting $\delta\rightarrow 0$, we see from
(\ref{eq:1e17}) that
\[
   \left| \int_{\bbR^{n_1+m}}
   \frac{f_1(y_1)   h(x)  dx dy_1 }{(|x_1-y_{11}| + |y_{12}| + | x |  )^{\lambda}}\right|
   \lesssim \|f_1\|_{L^{p_1}}   \|h\|_{L^{q'}}.
\]
Note that $|x_1-y_{11}| + |y_{12}| + | x | \approx |x| + |y_1|$.
The above inequality is equivalent to
\[
   \left| \int_{\bbR^{n_1+m}}
   \frac{f_1(y_1)   h(x)  dx dy_1 }{(|y_1| + | x |  )^{\lambda}}\right|
   \lesssim \|f_1\|_{L^{p_1}}   \|h\|_{L^{q'}},
   \quad \forall f_1\in L^{p_1},\, h\in L^{q'}.
\]
By Lemma~\ref{Lm:L1}, $p_1 \le q$.

(ii)\,\,   $p_1=\infty$.

In this case, $1<p_2<\infty$.
by setting $f_1 \equiv 1$, we see from
(\ref{eq:1e17}) that
\begin{equation}\label{eq:1e18}
   \left| \int_{\bbR^{n_2+m}}
   \frac{f_2(y_2)   h(x)  dx dy_2 }{(|x-y_{21}| + | y_{22} |
      )^{n_2/p'_2 + m/q}}\right|
   \lesssim \|f_2\|_{L^{p_2}}   \|h\|_{L^{q'}}.
\end{equation}
By Theorem~\ref{thm:linear}, we have
$p_2 < q$.

\subsection{The case $0<r_1, r_2<m$}

Since $I_{\lambda,D}$ is bounded, we have
\begin{equation}\label{eq:e23}
 \left\|
  \int_{\bbR^{n_1+n_2}}
    \frac{f_1(y_1) f_2(y_2)dy_1 dy_2 }
    {(|D_1x-y_1| + |D_2x-y_2|)^{\lambda}}
    \right\|_{L_x^q}
   \lesssim \|f_1\|_{L^{p_1}} \|f_2\|_{L^{p_2}}.
\end{equation}
Denote $D = \binom{D_1}{D_2}$.
Since $r_1=\rank (D_1)<m$, there is some $m\times m$ invertible matrix
$Q_1$ such that the last $m-r_1$ columns of $D_1Q_1$ are zero vectors.
On the other hand, since $\rank(D) = m$, the last $m-r_1$ columns of
$D_2Q_1$ must be linearly independent.
Hence there exist $1\le j_1 <\ldots<j_{r_1+r_2-m}\le r_1 $
such that the $j_1$-th, $\ldots$, the $j_{r_1+r_2-m}$-th
columns and the last
$m-r_1$ columns of $D_2Q_1$
are linearly independent.
Consequently, there is some $m\times m$ invertible matrix
$Q_2$ such that the first $m-r_2$ columns of $D_2Q_1Q_2$ are zero vectors
and
the last $m-r_1$ columns of $D_1Q_1Q_2$ are zero vectors.
That is,
\[
  \begin{pmatrix}
  D_1 \\
  D_2
  \end{pmatrix}
  Q_1 Q_2  =
  \begin{pmatrix}
  *  & *  & 0_{n_1 \times (m-r_1)} \\
  0_{n_2\times (m-r_2)} & *   & *
  \end{pmatrix}.
\]
Since $\rank(D_iQ_1Q_2)=\rank(D_i)$, there exist
$n_1\times n_1$ invertible matrix $P_1$
and
$n_2\times n_2$ invertible matrix $P_2$
such that
\[
P_1 D_1 Q_1 Q_2 =
  \begin{pmatrix}
  I_{r_1}   & 0  \\
    0  & 0
    \end{pmatrix}
    \quad
    and
    \quad
 P_2 D_2 Q_1Q_2=
  \begin{pmatrix}
  0    & I_{r_2}\\
    0 & 0
  \end{pmatrix}.
\]
Note that
\[
  |D_1x-y_1| + |D_2x-y_2|
  \approx
  |P_1D_1x-P_1y_1| + |P_2D_2x-P_2y_2|.
\]
By replacing $f_1(P_1\cdot)$ and $f_2(P_2\cdot)$
for $f_1$ and $f_2$ respectively and
a change of variables of the form
$(x,y_1,y_2)\rightarrow (Q_1Q_2x, P_1^{-1}y_1, P_2^{-1}y_2)$,
(\ref{eq:e23}) turns out to be
\begin{align}
& \left\|
  \int_{\bbR^{n_1+n_2}}
    \frac{f_1(y_1) f_2(y_2)dy_1 dy_2 }
    {K(x,y_1,y_2)^{\lambda}}
    \right\|_{L_x^q}
   \lesssim \|f_1\|_{L^{p_1}} \|f_2\|_{L^{p_2}},\label{eq:r1r2}
\end{align}
where
\[
K(x,y_1,y_2) =  |x_1 - y_{11}| + |x_2  -  y_{12}| + |y_{13}|
     +  |x_2 - y_{21}| +  |x_3  -  y_{22}|
      + |y_{23}| ,
\]
$x=(x_1,x_2,x_3)$, $y_i=(y_{i1}, y_{i2}, y_{i3})$,
$x_1, y_{11}\in\bbR^{m-r_2}$,
$x_2, y_{12},y_{21} \in \bbR^{r_1+r_2-m}$,
$x_3, y_{22}\in \bbR^{m-r_1}$,
$y_{13}\in\bbR^{n_1-r_1}$
and $y_{23}\in\bbR^{n_2-r_2}$.

Since $r_1, r_2<m$, similar arguments as that in
Subsection \ref{subsec:r1=0}
show that $q\ge p_i$, $i=1,2$.  Hence $q\ge \max\{p_1,p_2\}$.

We prove the rest in two cases.

(i)\,\, $\min\{p_1,p_2\}=1$.

We show that $q>p_2$ when $p_1=1$, $r_2=n_2$ and $r_1+r_2=m$.
In this case,
$K(x,y_1,y_2) = |x_1-y_{11}| + |y_{13}| + |x_3-y_2|$
and (\ref{eq:r1r2}) becomes
\[
  \left\|
  \int_{\bbR^{n_1+n_2}}
    \frac{f_1(y_1) f_2(y_2)dy_1 dy_2 }
    {(|x_1-y_{11}| + |y_{13}| + |x_3-y_2|)^{\lambda}}
    \right\|_{L_x^q}
   \lesssim \|f_1\|_{L^{p_1}} \|f_2\|_{L^{p_2}} .
\]
Setting $f_1 = (1/\delta^{n_1})\chi^{}_{\{|y_1|\le \delta\}}$
and letting $\delta\rightarrow 0$, we get
\[
  \left\|
  \int_{\bbR^{ n_2}}
    \frac{ f_2(y_2)  dy_2 }
    {(|x_1 |  + |x_3-y_2|)^{n_2/p'_2 + m/q}}
    \right\|_{L_x^q}
   \lesssim   \|f_2\|_{L^{p_2}},
\]
which is equivalent to
\[
  \left\|
  \int_{\bbR^{m}}
    \frac{ h(x_1,x_3) dx_1dx_3 }
    {(|x_1 |  + |x_3-y_2|)^{n_2/p'_2 + m/q}}
    \right\|_{L_{y_2}^{p'_2}}
   \lesssim   \|h\|_{L^{q'}}.
\]
By Theorem~\ref{thm:linear}, we have $p'_2 >q'$. Hence
$q>p_2$.

Similarly
we can prove that $q>p_1$ when $p_2=1$, $r_1=n_1$ and $r_1+r_2=m$.

(ii)\,\, $1<p_1,p_2<\infty$.

First we show that   $q>p_1$
when
$p_1=p_2$, $r_1+r_2=m$ and $n_i=r_i$ for some $i$.

Assume on the contrary that $q=p_1=p_2$.
Since $\lambda < n_1+n_2$,
we have $(n_1,n_2)\ne (r_1,r_2)$. Hence $n_1>r_1$ or $n_2>r_2$.
Without loss of generality, we assume that  $n_1=r_1$ and $n_2>r_2$.
In this case, $\lambda = r_1+r_2 + (n_2-r_2)/p'_2$.
We see from
(\ref{eq:r1r2}) that
\begin{equation}\label{eq:e24}
  \left\|
  \int_{\bbR^{r_1+m}}
    \frac{h(x)f_1(y_1)dx dy_1 }
    {(|x_1-y_1|   + |x_3-y_{22}| + |y_{23}| )^{\lambda}}
    \right\|_{L_{y_2}^{p'_2}}
   \lesssim \|f_1\|_{L^{p_1}} \|h\|_{L^{q'}}.
\end{equation}
Set $h = \chi^{}_{\{|x|\le 1\}}$
and $f_1 =    \chi^{}_{\{|y_1|\le 1\}}$.
For $\delta>0$ small enough and $|y_2|\le\delta$,
\begin{align*}
 &   \int_{\bbR^{r_1+m}}
    \frac{h(x)f_1(y_1)dx dy_1 }
    {(|x_1-y_1| + |x_3-y_{22}| + |y_{23}| )^{\lambda}}
         \\
 &\ge
   \int_{\substack{
   |y_1|\le \delta\\
   |x_1-y_1|\le |y_{23}| \\
     |x_3 - y_{22}|\le |y_{23}|
   }}
    \frac{h(x)f_1(y_1)dx dy_1 }
    {(|x_1-y_1| +|x_3-y_{22}| + |y_{23}| )^{\lambda}}
         \\
  &\gtrsim
     \frac{\delta^{r_1}}{|y_{23}|^{(n_2-r_2)/p'_2}},
\end{align*}
which contradicts (\ref{eq:e24}).

Next we show that $q>p_1$ when $p_1=p_2>2$,
$m=r_1+r_2$,
$n_1>r_1$ and $n_2>r_2$.

Again, assume that $q = p_1=p_2>2$. We have
\[
    \left\|
  \int_{\bbR^{r_1+m}}
    \frac{h(x)f_1(y_1)dx dy_1 }
    {(|x_1-y_{11}| +|y_{13}|  + |x_3-y_{22}| + |y_{23}| )^{\lambda}}
    \right\|_{L_{y_2}^{p'_2}}
   \lesssim \|f_1\|_{L^{p_1}} \|h\|_{L^{q'}}.
\]
Set
\[
  h(x) = \chi^{}_{\{|x|\le 1\}}(x)
  \quad\mbox{and}\quad
 f_1(y_1) =
  \frac{ \chi^{}_{\{|y_1|\le 1/2\}}(y_1)}
  { |y_{13}|^{(n_1-r_1)/p_1}
  (\log1/|y_{13}|)^{(1+\varepsilon)/p_1}}.
\]
For $\delta>0$ small enough and $|y_2|\le\delta$,
\begin{align*}
 &   \int_{\bbR^{n_1+m}}
    \frac{h(x)f_1(y_1)dx dy_1 }
    {(|x_1-y_{11}| + |y_{13}|+ |x_3-y_{22}| + |y_{23}| )^{\lambda}}
         \\
 &\ge
   \int_{\substack{
   |y_{11}|\le \delta\\
   |x_1-y_{11}|\le |y_{23}| \\
   |y_{23}|^2\le |y_{13}|\le |y_{23}|\\
     |x_3 - y_{22}|\le |y_{23}|
   }}
    \frac{h(x)f_1(y_1)dx dy_1 }
    {(|x_1-y_{11}| + |y_{13}|+|x_3-y_{22}| + |y_{23}| )^{\lambda}}
         \\
&  \qquad \times \frac{1}{|y_{13}|^{(n_1-r_1)/p_1}
  (\log1/|y_{13}|)^{(1+\varepsilon)/p_1}}     \\
  &\gtrsim
     \frac{\delta^{r_1}}{|y_{23}|^{(n_2-r_2)/p'_2}
     (\log1/|y_{23}|)^{(1+\varepsilon)/p_1}}.
\end{align*}
Since $p_1=p_2>2$, then there is some $\varepsilon>0$ such that
$p'_2(1+\varepsilon)/p_1<1$. Hence
\[
  \left\|   \int_{\bbR^{n_1+m}}
    \frac{h(x)f_1(y_1)dx dy_1 }
    {(|x_1-y_{11}| + |y_{13}|+ |x_3-y_{22}| + |y_{23}| )^{\lambda}}
    \right\|_{L_{y_2}^{p'_2}} =\infty.
\]
This completes the proof of the necessity.

\section{Proof of Theorem~\ref{thm:bilinear}: The Sufficiency}

In this section, we give the proof of the sufficiency
part in Theorem~\ref{thm:bilinear}.

First, we consider the case $q=\infty$.
In this case, $1<p_1, p_2<\infty$ and $1/p_1+1/p_2\ge 1$.
Note that
\begin{align*}
I_{\lambda,D}(f_1,f_2)(x)
&=   \int_{\bbR^{n_1+n_2}}
    \frac{f_1(y_1) f_2(y_2)dy_1 dy_2 }
    {(|D_1x-y_1| + |D_2x-y_2|)^{\lambda}} \\
& =   \int_{\bbR^{n_1+n_2}}
    \frac{f_1(y_1+D_1x) f_2(y_2+D_2x)dy_1 dy_2 }
    {(|y_1| + |y_2|)^{\lambda}}.
\end{align*}
By Lemma~\ref{Lm:L1}, we have
\[
  |I_{\lambda,D}(f_1,f_2)| \lesssim \|f_1\|_{L^{p_1}} \|f_2\|_{L^{p_2}},
  \qquad \forall x\in\bbR^m.
\]
Hence $I_{\lambda,D}$ is bounded from
$L^{p_1}\times L^{p_2}$ to $L^q$.

For the case $q<\infty$, we split the proof in several subsections.

\subsection{The case $r_1=r_2 = m$}

As in Subsection~\ref{subsec:r1r2=m}, we only need to prove that
for $f_1\in L^{p_1}$ and $f_2\in L^{p_2}$,
\begin{equation}\label{eq:4:e1}
    \left\| \int_{\bbR^{n_1+n_2}} \frac{f_1(y_{11},y_{12}) f_2(y_{21}, y_{22}) dy_1 dy_2}
   {(| x-y_{11}|+| x-y_{21}| + |y_{12}| + |y_{22}|)^{\lambda}} \right\|_{L^q}
 \lesssim \|f_1\|_{L^{p_1}} \|f_2\|_{L^{p_2}},
\end{equation}
where $y_i=(y_{i1},y_{i2})$,  $x,y_{i1}\in\bbR^m$,
$y_{i2}\in\bbR^{n_i-m}$, $i=1,2$.
There are two subcases.

(A1)\,\, $1<p_1, p_2<\infty$

We prove the conclusion in two  subcases.

(A1)(a) \,\,  $0<1/q<1/p_1 + 1/p_2$.

Choose $q_1, q_2$ such that
\[
  \frac{1}{q_1}  = \frac{1}{q}
                  \cdot \frac{1/p_1}{1/p_1 + 1/p_2}
                  \quad\mathrm{and}\quad
  \frac{1}{q_2}  = \frac{1}{q}
                  \cdot \frac{1/p_2}{1/p_1 + 1/p_2}.
\]
Then $1/q = 1/q_1 + 1/q_2$ and $1<p_i<q_i<\infty$.
Let
\[
  \lambda_i: = \frac{n_i}{p'_i} + \frac{m}{q_i},\qquad i=1,2.
\]
We have $\lambda = \lambda_1 + \lambda_2$
and $0<\lambda_i<n_i$.
By Theorem~\ref{thm:linear},
\[
  \left\| \int_{\bbR^{n_i}}
    \frac{f_i(y_i) dy_i}{(|x-y_{i1}| + |y_{i2}|)^{\lambda_i}}
  \right\|_{L^{q_i}}
  \lesssim \|f_i\|_{L^{p_i}}.
\]
Hence
\begin{align*}
&
  \left\| \int_{\bbR^{n_1+n_2}}
    \frac{f_1(y_1) f_2(y_2) dy_1 dy_2}
     {(|x-y_{11}| + |y_{12}| + |x-y_{21}| + |y_{22}|)^{\lambda}}
  \right\|_{L_x^q}\\
&\le
  \left\|
  \int_{\bbR^{n_1}}
    \frac{f_1(y_1) dy_1}{(|x-y_{11}| + |y_{12}|)^{\lambda_1}}
  \int_{\bbR^{n_2}}
    \frac{f_2(y_2) dy_2}{(|x-y_{21}| + |y_{22}|)^{\lambda_2}}
  \right\|_{L_x^q} \\
&\lesssim \|f_1\|_{L^{p_1}} \|f_2\|_{L^{p_2}}.
\end{align*}

(A1)(b)\,\,   $1/q = 1/p_1 + 1/p_2$.
In this case, $n_1, n_2 >m$ and $1/p_1 + 1/p_2\ge 1$.
Let us prove (\ref{eq:4:e1}).

For any $f_1\in L^{p_1}$ and $f_2\in L^{p_2}$, we have
\begin{align*}
&  \int_{ \bbR^{n_1+n_2}} \frac{|f_1(y_{11},y_{12}) f_2(y_{21}, y_{22})| dy_1
   dy_2 }
   {(| x-y_{11}|+| x-y_{21}| + |y_{12}| + |y_{22}|)^{\lambda}}
 \\
&\lesssim \int_{\bbR^{n_1+n_2-2m}}
\frac{ M_x f_1(x,y_{12}) M_xf_2(x, y_{22})  dy_{12}
   dy_{22} }
   {( |y_{12}| + |y_{22}|)^{(n_1-m)/p'_1 + (n_2-m)/p'_2}},
\end{align*}
where
\[
  M_x f_1(x,y_{12}) = \sup_{r>0} \frac{1}{r^m} \int_{|y_{11} - x|\le r}
     |f_1(y_{11}, y_{12})| dy_{11}
\]
is the partially maximal function of $f_1$,
and $M_xf_2(x,y_{22})$ is defined similarly.
It follows from Lemma~\ref{Lm:L1} that
\begin{align*}
   \int_{ \bbR^{n_1+n_2}}  \frac{|f_1(y_{11},y_{12}) f_2(y_{21}, y_{22})| dy_1
   dy_2 }
   {(| x-y_{11}|\!+\!| x-y_{21}|\! +\! |y_{12}| \!+\! |y_{22}|)^{\lambda}}
 &\!\lesssim \|M_x f_1(x,\cdot)\|_{L^{p_1}}
\|M_x f_2(x,\cdot)\|_{L^{p_2}}.
\end{align*}
Since $1/q = 1/p_1 + 1/p_2$, by H\"older's inequality, we get
\begin{align*}
& \left\| \int_{\bbR^{n_1+n_2}} \frac{ f_1(y_{11},y_{12}) f_2(y_{21}, y_{22})  dy_1
   dy_2 }
   {(| x-y_{11}|+| x-y_{21}| + |y_{12}| + |y_{22}|)^{\lambda}}
   \right\|_{L_x^q}
\lesssim    \|f_1\|_{L^{p_1}} \|f_2\|_{L^{p_2}}.
\end{align*}

(A2)\,\, $\#\{i:\, 1<p_i<\infty\}=1$.

Without loss of generality, assume that $1<p_2<\infty$.
There are two subcases.

(A2)(a)\,\, $p_1=1$.

In this case, $1<p_2\le q<\infty$. (\ref{eq:4:e1}) is equivalent to
\begin{equation}\label{eq:e15}
    \left\| \int_{\bbR^{n_2 + m}} \frac{ f_2(y_{21}, y_{22})h(x) dx dy_2}
   {(| x-y_{11}|+| x-y_{21}| + |y_{12}|
     + |y_{22}|)^{n_2/p'_2+m/q}} \right\|_{L_{y_1}^{\infty}}
    \!\!  \lesssim  \|f_2\|_{L^{p_2}} \|h\|_{L^{q'}}.
\end{equation}
Note that for any $y_1\in\bbR^{n_1}$,
\begin{align*}
&\int_{\bbR^{n_2 + m}} \frac{ |f_2(y_{21}, y_{22})h(x)| dx dy_2}
   {(| x-y_{11}|+| x-y_{21}| + |y_{12}|+|y_{22}|
      )^{n_2/p'_2+m/q}}\\
&\lesssim
   \int_{\bbR^{2m}} \frac{ \|f_2(y_{21}, \cdot)\|_{L^{p_2}} |h(x)| dx dy_{21}}
   {(| x-y_{11}|+| x-y_{21}|
      )^{m/p'_2+m/q}} \\
&\approx
   \int_{\bbR^{2m}} \frac{ \|f_2(y_{21}, \cdot)\|_{L^{p_2}} |h(x)| dx dy_{21}}
   {(| x-y_{11}|+| y_{11}-y_{21}|
      )^{m/p'_2+m/q}} \\
&=
   \int_{\bbR^{2m}} \frac{ \|f_2(y_{21}+y_{11}, \cdot)\|_{L^{p_2}}
   |h(x+y_{11})| dx dy_{21}}
   {(| x |+|  y_{21}|
      )^{m/p'_2+m/q}} \\
&\lesssim  \|f_2(\cdot+y_{11},\cdot)\|_{L^{p_2}} \|h(\cdot+y_{11})\|_{L^{q'}},
\end{align*}

where we use
Lemma~\ref{Lm:L1} in the last step.
Hence (\ref{eq:e15}) is true.

(A2)(b)\,\, $p_1=\infty$.

In this case, $p_2 < q<\infty$ and $\lambda = n_1+n_2/p'+m/q$.
We have
\begin{align*}
 & \int_{\bbR^{n_1+n_2}} \frac{| f_1(y_{11}, y_{12})
     f_2(y_{21}, y_{22})| dy_1 dy_2}
   {(| x-y_{11}|+| x-y_{21}| + |y_{12}|
     + |y_{22}|)^{\lambda }}   \\
&\lesssim \|f_1\|_{L^{\infty}}
  \int_{\bbR^{n_2}} \frac{|
     f_2(y_{21}, y_{22})|   dy_2}
   {(  | x-y_{21}|
     + |y_{22}|)^{n_2/p'+m/q }}  .
\end{align*}
Now the conclusion follows from Theorem~\ref{thm:linear}.

\subsection{The case $r_1= 0$ and $r_2 = m$}

In this case, we need to show that
\begin{equation}\label{eq:e16:a}
 \left\| \int_{\bbR^{n_1+n_2}}
   \frac{f_1(y_1) f_2(y_2)dy_1 dy_2}{(|y_1| + | x-y_{21}| + |y_{22}|)^{\lambda}}
   \right\|_{L_x^q}
   \lesssim \|f_1\|_{L^{p_1}} \|f_2\|_{L^{p_2}},
\end{equation}
where $y_2 = (y_{21}, y_{22})$, $y_{21}\in\bbR^m$
and $y_{22}\in\bbR^{n_2-m}$.
We prove the conclusion in three subcases.

(B1)  $p_2 = 1$.

In this case, $1<p_1 \le q < \infty$.
It suffices to show that
\begin{equation}\label{eq:e21}
 \left\| \int_{\bbR^{n_1+ m}}
   \frac{f_1(y_1)   h(x)  dx dy_1  }{(|y_1| + | x-y_{21}| + |y_{22}|)^{\lambda}}\right\|_{L_{y_2}^{\infty}}
   \lesssim \|f_1\|_{L^{p_1}}  \|h\|_{L^{q'}}.
\end{equation}
Observe that
\begin{align*}
\int_{\bbR^{n_1+ m}}
   \frac{|f_1(y_1)   h(x)|  dx dy_1  }{(|y_1| + | x-y_{21}| + |y_{22}|)^{\lambda}}
 &  =
\int_{\bbR^{n_1+ m}}
   \frac{|f_1(y_1)   h(x+y_{21})|  dx dy_1  }{(|y_1| + | x| + |y_{22}|)^{n'_1/p_1 + m/q}}
    \\
 &\le
\int_{\bbR^{n_1+ m}}
   \frac{|f_1(y_1)   h(x+y_{21})|  dx dy_1  }{(|y_1| + | x| )^{n'_1/p_1 + m/q}}.
\end{align*}
By Lemma~\ref{Lm:L1}, we have
\[
  \int_{\bbR^{n_1+ m}}
   \frac{|f_1(y_1)   h(x)|  dx dy_1  }{(|y_1| + | x-y_{21}| + |y_{22}|)^{\lambda}}
\lesssim \|f_1\|_{L^{p_1}}  \|h\|_{L^{q'}},
\qquad \forall y_2\in\bbR^{n_2}.
\]
Hence (\ref{eq:e21}) is true.

(B2)\,\,  $ p_2>1$ and  $p_2<q<\infty$.

By H\"older's inequality, we have
\begin{align*}
\int_{\bbR^{n_1+ n_2}}
   \frac{|f_1(y_1)   f_2(y_2)|   dy_1 dy_2  }{(|y_1| + | x-y_{21}| + |y_{22}|)^{\lambda}}
&\lesssim
   \int_{\bbR^{n_2}}
   \frac{  \|f_1\|_{L^{p_1}}  |    f_2(y_2)|    dy_2  }
   {(  | x-y_{21}| + |y_{22}|)^{n_2/p'_2+m/q}}.
\end{align*}
Now we see from Theorem~\ref{thm:linear} that
(\ref{eq:e16:a}) is true.

(B3)\, $ p_2>1$ and $q=p_2$.

In this case, $n_2>m$, $1<p_1,p_2<\infty$ and $1/p_1 + 1/p_2 \ge 1$.
By Minkowski's and Young's inequalities, we have
\begin{align*}
&\left\|\int_{\bbR^{n_1+ n_2}}
   \frac{f_1(y_1)   f_2(y_2)   dy_1 dy_2  }{(|y_1| + | x-y_{21}| + |y_{22}|)^{\lambda}}
   \right\|_{L_x^q} \\
&\lesssim
\int_{\bbR^{n_1+ n_2-m}}
   \frac{f_1(y_1)  \| f_2(\cdot,y_{22})\|_{L^{p_2}}   dy_1 dy_{22}  }{(|y_1| +   |y_{22}|)^{n_1/p'_1 + (n_2-m)/p'_2}}.
\end{align*}
Now it follows from   Lemma~\ref{Lm:L1} that
\begin{align*}
 \left\|\int_{\bbR^{n_1+ n_2}}
   \frac{f_1(y_1)   f_2(y_2)   dy_1 dy_2  }{(|y_1| + | x-y_{21}| + |y_{22}|)^{\lambda}}
   \right\|_{L_x^q}
 \lesssim \|f_1\|_{L^{p_1}} \|f_2\|_{L^{p_2}} .
\end{align*}

\subsection{The case $0<r_1<r_2=m$ or $0<r_2<r_1=m$}

We consider only the case $0<r_1<r_2=m$.
As in the necessity part, it suffices to show that
\begin{equation}\label{eq:1e16:a}
 \left\| \int_{\bbR^{n_1+n_2}}
   \frac{f_1(y_1) f_2(y_2)dy_1 dy_2}{(|x_1-y_{11}| + |y_{12}| + | x-y_{21}| + |y_{22}|)^{\lambda}}
   \right\|_{L_x^q}
   \lesssim \|f_1\|_{L^{p_1}} \|f_2\|_{L^{p_2}},
\end{equation}
where
$x=(x_1,x_2)$, $x_1\in\bbR^{r_1}$,
$x_2 \in\bbR^{m-r_1}$,
$y_{i1}\in\bbR^{r_i}$ and $y_{i2}\in \bbR^{n_i-r_i}$,
$i=1,2$.

We prove the conclusion in four subcases.

(C1)  $p_2 = 1$.

In this case, $1<p_1 \le q < \infty$.  (\ref{eq:1e16:a}) is equivalent to
\begin{equation}\label{eq:1e21}
 \left\| \int_{\bbR^{n_1+ m}}
   \frac{f_1(y_1)   h(x)  dx dy_1  }
   {(|x_1-y_{11}| + |y_{12}| + | x-y_{21}| + |y_{22}|)^{\lambda}}\right\|_{L_{y_2}^{\infty}}
   \lesssim \|f_1\|_{L^{p_1}}  \|h\|_{L^{q'}}.
\end{equation}
Observe that
\begin{align*}
&\int_{\bbR^{n_1+ m}}
   \frac{|f_1(y_1)   h(x)|  dx dy_1  }{(|x_1-y_{11}| + |y_{12}|  + | x-y_{21}| + |y_{22}|)^{\lambda}}\\
 &  =
\int_{\bbR^{n_1+ m}}
   \frac{|f_1(y_1+(y_{21}^{(1)},\ldots,
   y_{21}^{(r_1)}, 0, \ldots, 0
   ))   h(x+y_{21})|  dx dy_1  }{(|x_1-y_{11}| + |y_{12}|  + | x| + |y_{22}|)^{n'_1/p_1 + m/q}}
    \\
 &\lesssim
\int_{\bbR^{n_1+ m}}
   \frac{|f_1(y_1+(y_{21}^{(1)},\ldots,
   y_{21}^{(r_1)}, 0, \ldots, 0
   ))   h(x+y_{21})|
       dx dy_1  }{(|y_1|  + | x| )^{n'_1/p_1 + m/q}}.
\end{align*}
By Lemma~\ref{Lm:L1}, we have
\[
\int_{\bbR^{n_1+ m}}
   \frac{|f_1(y_1)   h(x)|  dx dy_1  }{(|x_1-y_{11}| + |y_{12}|  + | x-y_{21}| + |y_{22}|)^{\lambda}}
   \lesssim \|f_1\|_{L^{p_1}}  \|h\|_{L^{q'}},
  \forall y_2\in\bbR^{n_2}.
\]
Hence (\ref{eq:1e21}) is true.

(C2)\,\, $p_1 = 1$.

In this case, $1<p_2 \le q <\infty$.
It suffices to show that for any $f_2\in L^2$
and $h\in L^{q'}$,
\begin{align}
 \left| \int_{\bbR^{n_2+ m}}
   \frac{f_2(y_2)   h(x)  dx dy_2  }
   {(|x_1-y_{11}| + |y_{12}| + | x-y_{21}| + |y_{22}|)^{\lambda}}
   \right |
   \lesssim \|f_2\|_{L^{p_2}}  \|h\|_{L^{q'}},
\quad \forall y_1\in\bbR^{n_1}. \label{eq:e:e2}
\end{align}
Denote $y_{21}=(y_{211}, y_{212})$, where
$y_{211}\in\bbR^{r_1}$, $y_{212}\in \bbR^{m-r_1}$.
We rewrite the above inequality as
\begin{align*}
&  \left| \int_{\bbR^{n_2+ m}}
   \frac{f_2(y_2)   h(x)  dx dy_2  }
   {(|x_1-y_{11}| \!+\! |y_{12}| \!+\!| x_1-y_{211}|\!+\!|x_2-y_{212}|
    \!+\! |y_{22}|)^{\lambda}}
   \right |
   \lesssim \|f_2\|_{L^{p_2}}  \|h\|_{L^{q'}},
\,\, \forall y_1\in\bbR^{n_1}.
\end{align*}
Using Young's inequality when computing the integration with respect to
$dx_2 dy_{212}$, we get
\begin{align*}
& \left| \int_{\bbR^{n_2+ m}}
   \frac{f_2(y_2)   h(x)  dx dy_2  }
   {(|x_1-y_{11}| \!+\! |y_{12}| \!+\!| x_1-y_{211}|\!+\!|x_2-y_{212}|
    \!+\! |y_{22}|)^{\lambda}}
   \right |
   \\
&\lesssim
  \int_{\bbR^{2r_1+ n_2-m}}
   \frac{\|f_2(y_{211},\cdot,y_{22})\|_{L_{y_{212}}^{p_2}}
       \|h(x_1,\cdot)\|_{L_{x_2}^{q'}}  dx_1 dy_{211}dy_{22}  }
   {(|x_1-y_{11}| \!+\! |y_{12}| \!+\!| x_1-y_{211}|
    \!+\! |y_{22}|)^{(n_2+r_1-m)/p'_2 + r_1/q}}
   \\
&=
  \int_{\bbR^{2r_1+ n_2-m}}
   \frac{\|f_2(y_{211}+y_{11},\cdot,y_{22})\|_{L_{y_{212}}^{p_2}}
       \|h(x_1+y_{11},\cdot)\|_{L_{x_2}^{q'}}  dx_1 dy_{211}dy_{22}  }
   {(|x_1 | + |y_{12}| +| x_1-y_{211}|
    + |y_{22}|)^{(n_2+r_1-m)/p'_2 + r_1/q}}\\
&\lesssim
  \int_{\bbR^{2r_1+ n_2-m}}
   \frac{\|f_2(y_{211}+y_{11},\cdot,y_{22})\|_{L_{y_{212}}^{p_2}}
       \|h(x_1+y_{11},\cdot)\|_{L_{x_2}^{q'}}  dx_1 dy_{211}dy_{22}  }
   {(|x_1 |  +|  y_{211}|
    + |y_{22}|)^{(n_2+r_1-m)/p'_2 + r_1/q}}\\
&\lesssim \|f_2\|_{L^{p_2}} \|h\|_{L^{q'}},
\end{align*}
where we use Lemma~\ref{Lm:L1} in the last step.
Hence (\ref{eq:e:e2}) is true.

(C3)\,\, $p_1 = \infty$.

In this case, $1<p_2 < q <\infty$.
We have
\begin{align*}
&
\left|\int_{\bbR^{n_1+ n_2}}
   \frac{f_1(y_1)   f_2(y_2)   dy_1 dy_2  }
   {(|x_1-y_{11}| + |y_{12}| + | x-y_{21}| + |y_{22}|)^{\lambda}}
   \right|
     \\
&\lesssim
 \int_{\bbR^{n_1 }}
   \frac{\|f_1\|_{L^{\infty}}   f_2(y_2)   dy_2  }
   {(  | x-y_{21}| + |y_{22}|)^{\lambda-n_1}}.
\end{align*}
Now the conclusion follows from Theroem~\ref{thm:linear}.

(C4)\,\, $1<p_1, p_2<\infty$ and $p_2 \le q <\infty$.

Using H\"older's inequality when computing
the integration with respect to
$y_{12}$ and $y_{22}$, we get
\begin{align*}
&
\left|\int_{\bbR^{n_1+ n_2}}
   \frac{f_1(y_1)   f_2(y_2)   dy_1 dy_2  }
   {(|x_1-y_{11}| + |y_{12}| + | x-y_{21}| + |y_{22}|)^{\lambda}}
   \right|
     \\
&\lesssim
 \int_{\bbR^{r_1+m }}
   \frac{\|f_1(y_{11},\cdot)\|_{L^{p_1}}   \|f_2(y_{21},\cdot)\|_{L^{p_2}}   dy_{11} dy_{21}}
   {(  |x_1-y_{11}|  + | x-y_{21}|  )^{r_1/p'_1 + m/p'_2 + m/q}}.
\end{align*}
Denote   $y_{21} = (y_{211}, y_{212})$,
where $y_{211}\in\bbR^{r_1}$ and $y_{212}\in\bbR^{m-r_1}$.
By Young's inequality, we get
\begin{align*}
&
\left\|\int_{\bbR^{n_1+ n_2}}
   \frac{f_1(y_1)   f(y_2)   dy_1 dy_2  }
   {(|x_1-y_{11}| + |y_{12}| + | x-y_{21}| + |y_{22}|)^{\lambda}}
   \right\|_{L_{x_2}^q}
     \\
&\lesssim
 \int_{\bbR^{2r_1 }}
   \frac{\|f_1(y_{11},\cdot)\|_{L^{p_1}}   \|f(y_{211},\cdot)\|_{L^{p_2}}   dy_{11} dy_{211}}
   {(  |x_1-y_{11}| + |x_1 - y_{211}|  )^{r_1/p'_1 +r_1/p'_2+ r_1/q}}.
\end{align*}
 Now the conclusion follows from Proposition~\ref{prop:P1}.

\subsection{The case $0<r_1, r_2<m$}

As in the necessity part, we only need to show that
\begin{align*}
& \left\|
  \int_{\bbR^{n_1+n_2}}
    \frac{f_1(y_1) f_2(y_2)dy_1 dy_2 }
    {K(x,y_1,y_2)^{\lambda}}
    \right\|_{L_x^q}
   \lesssim \|f_1\|_{L^{p_1}} \|f_2\|_{L^{p_2}},
\end{align*}
where
\[
K(x,y_1,y_2) =  |x_1 - y_{11}| + |x_2  -  y_{12}| + |y_{13}|
     +  |x_2 - y_{21}| +  |x_3  -  y_{22}|
      + |y_{23}| ,
\]
$x=(x_1,x_2,x_3)$, $y_i=(y_{i1}, y_{i2}, y_{i3})$,
$x_1, y_{11}\in\bbR^{m-r_2}$,
$x_2, y_{12},y_{21} \in \bbR^{r_1+r_2-m}$,
$x_3, y_{22}\in \bbR^{m-r_1}$,
$y_{13}\in\bbR^{n_1-r_1}$
and $y_{23}\in\bbR^{n_2-r_2}$.

There are five subcases.

(D1) $p_1=1$ and $r_2<n_2$.

In this case, $p_2\le q<\infty$.
Observe that
\begin{align*}
K(x,y_1,y_2)
&\ge   |x_1- y_{11}| +
     |x_2-y_{21}| + |x_3 - y_{22}| + |y_{23}|.
\end{align*}
We have
\begin{align*}
& \left\|
  \int_{\bbR^{n_2}}
    \frac{  f_2(y_2)  dy_2 }
    {K(x,y_1,y_2)^{\lambda}}
    \right\|_{L_{(x_2,x_3)}^q} \\
&\le  \left\|
  \int_{\bbR^{n_2}}
    \frac{  f_2(y_2)  dy_2 }
    {(|x_1 - y_{11}|        +  |x_2 - y_{21}| +  |x_3  -  y_{22}|
      + |y_{23}| )^{\lambda}}
    \right\|_{L_{(x_2,x_3)}^q} \\
&\lesssim
  \int_{\bbR^{n_2-r_2}}
         \frac{ \|f_2(\cdot, y_{23})\|_{L^{p_2}}  dy_{23} }
    {(|x_1 - y_{11}|
      + |y_{23}| )^{(n_2-r_2)/p'_2 + (m-r_2)/q}},
\end{align*}
where we use Young's inequality in the last step.
Hence
\begin{align*}
 \left\|
  \int_{\bbR^{n_2}}
 \!   \frac{  f_2(y_2)  dy_2 }
    {K(x,y_1,y_2)^{\lambda}}
    \right\|_{L_x^q}
&\lesssim
  \left\|  \int_{\bbR^{n_2-r_2}}
         \frac{ \|f_2(\cdot, y_{23})\|_{L^{p_2}}  dy_{23} }
    {(|x_1 - y_{11}|
     \! +\! |y_{23}| )^{(n_2-r_2)/p'_2 + (m-r_2)/q}}
      \right\|_{L_{x_1}^q}.
\end{align*}
By Lemma~\ref{Lm:L1},
\[
 \left\|
  \int_{\bbR^{n_2}}
 \!   \frac{  f_2(y_2)  dy_2 }
    {K(x,y_1,y_2)^{\lambda}}
    \right\|_{L_x^q}
    \lesssim
   \|f_2\|_{L^{p_2}},\qquad \forall y_1\in\bbR^{n_1}.
\]
It follows from Minkowski's inequality that
\begin{align}
 \left\|
  \int_{\bbR^{n_1+n_2}}
    \frac{ f_1(y_1) f_2(y_2)  dy_1 dy_2 }
    {K(x,y_1,y_2)^{\lambda}}
    \right\|_{L_x^q}
\lesssim \|f_1\|_{L^{p_1}} \|f_2\|_{L^{p_2}}. \label{eq:e26}
\end{align}

(D2) $p_1=1$, $r_2=n_2$ and $r_1+r_2>m$.

In this case, $p_2 \le q <\infty$ and
\begin{align*}
K(x,y_1,y_2)
&= |x_1- y_{11}| + |x_2 - y_{12}| + |y_{13}|
   + |x_2-y_{21}| + |x_3 - y_{22}| \\
&\ge   |x_1- y_{11}| + |x_2 - y_{12}|
   + |x_2-y_{21}| + |x_3 - y_{22}|\\
&\approx    |x_1- y_{11}| + |x_2 - y_{12}|
  + |y_{12} - y_{21}| + |x_3-y_{22}| .
\end{align*}
Hence
\begin{align*}
& \left\|
  \int_{\bbR^{n_2}}
    \frac{  f_2(y_2)  dy_2 }
    {K(x,y_1,y_2)^{\lambda}}
    \right\|_{L_{x_3}^q} \\
&\le  \left\|
  \int_{\bbR^{n_2}}
    \frac{  f_2(y_2)  dy_2 }
    {(|x_1- y_{11}| + |x_2 - y_{12}|
   + |x_3-y_{22}| + |y_{12} - y_{21}| )^{\lambda}}
    \right\|_{L_{x_3}^q} \\
&\lesssim
  \int_{\bbR^{n_2+r_1-m}}
         \frac{ \|f_2(y_{21},\cdot)\|_{L^{p_2}}  dy_{21} }
    {(|x_1- y_{11}| + |x_2 - y_{12}|
      + |y_{12} - y_{21}| )^{(n_2+r_1-m)/p'_2 +  r_1 /q}}\\
&=
  \int_{\bbR^{n_2+r_1- m}}
         \frac{ \|f_2( y_{21}+y_{12}, \cdot )\|_{L^{p_2}}  dy_{21} }
    {(|x_1- y_{11}| + |x_2 - y_{12}|
      + | y_{21}| )^{(n_2+r_1-m)/p'_2 +  r_1 /q}}.
\end{align*}
It follows from Lemma~\ref{Lm:L1} that
\[
\left\|
  \int_{\bbR^{n_2}}
    \frac{  f_2(y_2)  dy_2 }
    {K(x,y_1,y_2)^{\lambda}}
    \right\|_{L_x^q}
    \lesssim
    \|f_2\|_{L^{p_2}},\quad \forall y_1\in\bbR^{n_1}.
\]
Thus (\ref{eq:e26}) is true.

(D3) $p_1=1$, $r_2=n_2$ and  $r_1+r_2=m$.

In this case, $p_2<q<\infty$ and
\begin{align*}
K(x,y_1,y_2)
&= |x_1- y_{11}| +  |y_{13}|
   + |x_3-y_2|
 \ge  |x_1- y_{11}|    + |x_3-y_2|  .
\end{align*}
Hence
\[
  \int_{\bbR^{n_2}}
    \frac{  |f_2(y_2)|  dy_2 }
    {K(x,y_1,y_2)^{\lambda}}
  \le  \int_{\bbR^{n_2}}
    \frac{  |f_2(y_2)|  dy_2 }
    {(|x_1- y_{11}|    + |x_3-y_2| )^{\lambda}}.
\]
It follows from   Minkowski's inequality that
\begin{align*}
 \left\|
  \int_{\bbR^{n_2}}
    \frac{  f_2(y_2)  dy_2 }
    {K(x,y_1,y_2)^{\lambda}}
    \right\|_{L_{x_1}^q}
&\le  \int_{\bbR^{n_2}}\frac{  |f_2(y_2)|  dy_2 }
    {  |x_3-y_2|^{r_2/p'_2 + r_2/q}}.
\end{align*}
By Theorem~\ref{thm:linear}, we get
\[
 \left\|
   \int_{\bbR^{n_2}}\frac{  |f_2(y_2)|  dy_2 }
    {  |x_3-y_2|^{r_2/p'_2 + r_2/q}}
    \right\|_{L_{x_3}^q}
    \lesssim
    \|f_2\|_{L^{p_2}}.
\]
Hence
\[
 \left\|
  \int_{\bbR^{n_2}}
    \frac{  f_2(y_2)  dy_2 }
    {K(x,y_1,y_2)^{\lambda}}
    \right\|_{L_{x}^q}
=\left\| \left\|
  \int_{\bbR^{n_2}}
    \frac{  f_2(y_2)  dy_2 }
    {K(x,y_1,y_2)^{\lambda}}
    \right\|_{L_{x_1}^q}
  \right\|_{L_{x_3}^q}
  \lesssim \|f_2\|_{L^{p_2}},
  \quad \forall y_1\in\bbR^{n_1}.
\]
Therefore, (\ref{eq:e26}) is true.

(D4)\, $p_2=1$.

Similarly to Cases (D1) - (D3) we get the conclusion.

(D5)\, $1<p_1, p_2<\infty$ and $\max\{p_1, p_2\}\le q<\infty$

There are two subcases.

(D5)(a)\,\, $r_1+r_2>m$.

Recall that
\[
  K(x,y_1,y_2)
 = |x_1\!-\! y_{11}| + |x_2 \!-\! y_{12}| + |y_{13}|
   + |x_2 \!-\! y_{21}| + |x_3 \!-\! y_{22}| + |y_{23}| .
\]
Using H\"older's inequality when computing
the integrations $d y_{13}$ and $dy_{23}$, we get
\begin{align}
& \left|
  \int_{\bbR^{n_1+n_2}}
    \frac{ f_1(y_1) f_2(y_2) dy_1 dy_2 }
    {K(x,y_1,y_2)^{\lambda}}
    \right| \nonumber \\
&\lesssim
  \int_{\bbR^{n_1+n_2}}
    \frac{ \|f_1(y_{11},y_{12},\cdot)\|_{L^{p_1}}
      \| f_2(y_{21}, y_{22}, \cdot)\|_{L^{p_2}} dy_{11} dy_{12} dy_{21} dy_{22} }
    {(|x_1\!-\! y_{11}| + |x_2 \!-\! y_{12}|
  + |x_2 \!-\! y_{21}| + |x_3 \!-\! y_{22}| )^{r_1/p'_1 + r_2/p'_2+m/q}}.
   \label{eq:e30}
\end{align}
We see from Young's inequality
that
\begin{align*}
 &  \left\|
  \int_{\bbR^{n_1+n_2}}
    \frac{ f_1(y_1) f_2(y_2) dy_1 dy_2 }
    {K(x,y_1,y_2)^{\lambda}}
    \right\|_{L_{(x_1,x_3)}^q}\\
  &\lesssim
  \int_{\bbR^{2(r_1+r_2-m)}}
  \frac{\|f_1(\cdot,y_{12},\cdot)\|_{L^{p_1}}
  \|f_2(y_{21},\cdot)\|_{L^{p_2}}dy_{12} dy_{21}}
  {(|x_2-y_{12}| + |x_2 - y_{21}|)^{(r_1+r_2-m)(1/p'_1+1/p'_2+1/q)}}.
\end{align*}
By Proposition~\ref{prop:P1}, we get the conclusion as desired.

(D5)(b)\,\, $r_1+r_2=m$.

In this case, the variables $x_2, y_{12}, y_{21}$ do not exist.
If $p_1<p_2$,
 we see from  H\"older's and Young's inequalities that
\begin{align*}
   \left\|
  \int_{\bbR^{n_1+n_2}}
    \frac{ f_1(y_1) f_2(y_2) dy_1 dy_2 }
    {K(x,y_1,y_2)^{\lambda}}
    \right\|_{L_{x_3}^q}
&\lesssim
  \int_{\bbR^{n_1}}
     \frac{|f_1(y_1)| \cdot \|f_2\|_{L^{p_2}}dy_1}
     {(|x_1-y_{11}| + |y_{13}|)^{n_1/p'_1 + r_1/q}}.
\end{align*}
By Theorem~\ref{thm:linear}, we get
\[
     \left\|
  \int_{\bbR^{n_1+n_2}}
    \frac{ f_1(y_1) f_2(y_2) dy_1 dy_2 }
    {K(x,y_1,y_2)^{\lambda}}
    \right\|_{L_x^q}
  \lesssim \|f_1\|_{L^{p_1}}\|f_2\|_{L^{p_2}} .
\]

If $p_2<p_1$ or $p_1=p_2<q$, with similar arguments we get the conclusion.

If $q=p_1=p_2$, then we have $q\le 2$, $n_1>r_1$ and $n_2>r_2$.
By Young's inequality, we get
\begin{align*}
&  \left\|
  \int_{\bbR^{n_1+n_2}}
    \frac{ f_1(y_1) f_2(y_2) dy_1 dy_2 }
    {K(x,y_1,y_2)^{\lambda}}
    \right\|_{L_x^q}\\
&\lesssim
  \int_{\bbR^{n_1+n_2-r_1-r_2}}
     \frac{\|f_1(\cdot,y_{13})\|_{L^{p_1}}
     \|f_2(\cdot, y_{23})\|_{L^{p_2}}dy_{13} dy_{23}}
     {( |y_{13}| + |y_{23}|)^{(n_1-r_1)/p'_1 + (n_2-r_2)/p'_2}}.
\end{align*}
Now the conclusion follows from Lemma~\ref{Lm:L1}.
This completes the proof.


\begin{thebibliography}{10}

\bibitem{Benyi2015}
A.~B\'{e}nyi, W.~Dami\'{a}n, K.~Moen, and R.~H. Torres.
\newblock Compactness properties of commutators of bilinear fractional
  integrals.
\newblock {\em Math. Z.}, 280(1-2):569--582, 2015.

\bibitem{Cao2019}
M.~Cao and Q.~Xue.
\newblock A revisit on commutators of linear and bilinear fractional integral
  operator.
\newblock {\em Tohoku Math. J. (2)}, 71(2):303--318, 2019.

\bibitem{Chaffee2015}
L.~Chaffee and R.~H. Torres.
\newblock Characterization of compactness of the commutators of bilinear
  fractional integral operators.
\newblock {\em Potential Anal.}, 43(3):481--494, 2015.

\bibitem{Chen2011}
J.~Chen and D.~Fan.
\newblock A bilinear fractional integral on compact {L}ie groups.
\newblock {\em Canad. Math. Bull.}, 54(2):207--216, 2011.

\bibitem{Chirst1987}
M.~Christ and J.-L. Journ\'{e}.
\newblock Polynomial growth estimates for multilinear singular integral
  operators.
\newblock {\em Acta Math.}, 159(1-2):51--80, 1987.

\bibitem{Coifman1975}
R.~R. Coifman and Y.~Meyer.
\newblock On commutators of singular integrals and bilinear singular integrals.
\newblock {\em Trans. Amer. Math. Soc.}, 212:315--331, 1975.

\bibitem{Grafakos1992}
L.~Grafakos.
\newblock On multilinear fractional integrals.
\newblock {\em Studia Math.}, 102(1):49--56, 1992.

\bibitem{Grafakos2014}
L.~Grafakos.
\newblock {\em Classical {F}ourier analysis}, volume 249 of {\em Graduate Texts
  in Mathematics}.
\newblock Springer, New York, third edition, 2014.

\bibitem{Grafakos2014m}
L.~Grafakos.
\newblock {\em Modern {F}ourier analysis}, volume 250 of {\em Graduate Texts in
  Mathematics}.
\newblock Springer, New York, third edition, 2014.

\bibitem{GrafakosKalton2001}
L.~Grafakos and N.~Kalton.
\newblock Some remarks on multilinear maps and interpolation.
\newblock {\em Math. Ann.}, 319(1):151--180, 2001.

\bibitem{GrafakosLynch2015}
L.~Grafakos and R.~G. Lynch.
\newblock Off-diagonal multilinear interpolation between adjoint operators.
\newblock {\em Comment. Math.}, 55(1):17--22, 2015.

\bibitem{Grafakos2002}
L.~Grafakos and R.~H. Torres.
\newblock Multilinear {C}alder\'on-{Z}ygmund theory.
\newblock {\em Adv. Math.}, 165(1):124--164, 2002.

\bibitem{Hoang2018}
C.~Hoang and K.~Moen.
\newblock Weighted estimates for bilinear fractional integral operators and
  their commutators.
\newblock {\em Indiana Univ. Math. J.}, 67(1):397--428, 2018.

\bibitem{KenigStein1999}
C.~E. Kenig and E.~M. Stein.
\newblock Multilinear estimates and fractional integration.
\newblock {\em Math. Res. Lett.}, 6(1):1--15, 1999.

\bibitem{Komori2017}
Y.~Komori-Furuya.
\newblock Notes on endpoint estimates for multilinear fractional integral
  operators.
\newblock {\em Proc. Amer. Math. Soc.}, 145(11):1015--1526, 2017.

\bibitem{Komori-Furuya2020}
Y.~Komori-Furuya.
\newblock Weighted estimates for bilinear fractional integral operators: a
  necessary and sufficient condition for power weights.
\newblock {\em Collect. Math.}, 71(1):25--37, 2020.

\bibitem{Lacey1997}
M.~Lacey and C.~Thiele.
\newblock {$L^p$} estimates on the bilinear {H}ilbert transform for
  {$2<p<\infty$}.
\newblock {\em Ann. of Math. (2)}, 146(3):693--724, 1997.

\bibitem{Lacey1999}
M.~Lacey and C.~Thiele.
\newblock On {C}alder\'{o}n's conjecture.
\newblock {\em Ann. of Math. (2)}, 149(2):475--496, 1999.

\bibitem{Li2017}
J.~Li and P.~Liu.
\newblock Bilinear fractional integral along homogeneous curves.
\newblock {\em J. Fourier Anal. Appl.}, 23(6):1465--1479, 2017.

\bibitem{LiSun2016}
K.~Li and W.~Sun.
\newblock Two weight norm inequalities for the bilinear fractional integrals.
\newblock {\em Manuscripta Math.}, 150(1-2):159--175, 2016.

\bibitem{Moen2014}
K.~Moen.
\newblock New weighted estimates for bilinear fractional integral operators.
\newblock {\em Trans. Amer. Math. Soc.}, 366(2):627--646, 2014.

\bibitem{Stein1970}
E.~M. Stein.
\newblock {\em Singular integrals and differentiability properties of
  functions}.
\newblock Princeton Mathematical Series, No. 30. Princeton University Press,
  Princeton, N.J., 1970.

\end{thebibliography}

\end{document}